\documentclass{article}
\usepackage[a4paper]{geometry}
\usepackage{graphicx} 
\usepackage{mathtools,amssymb,bm,amsthm}
\usepackage{ytableau}
\usepackage{algorithm2e}[ruled]
\usepackage{tikz}
\usepackage{tikz-cd}

\newcommand{\cone}[2]{\mathrm{Cone}_{#1}(#2)}

\newcommand{\modulo}{\ \mathrm{mod} \ }
\newcommand{\qinteger}[1]{(#1)_q}
\newcommand{\qfactorial}[1]{(#1)_q^!}
\newcommand{\qbinom}[2]{\binom{#1}{#2}_q}
\newcommand{\projbound}[1]{p(#1,\ell)}
\newcommand{\injbound}[1]{i(#1,\ell)}
\newcommand{\composition}[2]{\Lambda(#1,#2)} 
\newcommand{\flag}[1]{\mathcal{F}_{#1}} 
\newcommand{\leftcoset}[1]{\mathcal{D}_{#1}} 
\newcommand{\rightcoset}[1]{\mathcal{D}^{-1}_{#1}} 
\newcommand{\doublecosetrepr}[2]{\mathcal{D}_{#1,#2}} 
\newcommand{\orbit}[1]{\mathcal{O}_{#1}} 
\newcommand{\matrixwithfixedrowcol}[2]{\Theta_{#1,#2}}
\newcommand{\matrixwithfixedsum}[3]{\Theta(#1,#2)_{#3}}
\newcommand{\compositionsanszero}[2]{\Lambda_{>0}(#1,#2)}

\newcommand{\refinement}[2]{\mathbf{Ref}_{#1}(#2)}
\newcommand{\field}{\Bbbk} 
\newcommand{\ffield}{\mathbb{F}_{q}} 
\newcommand{\GL}[1]{\mathbf{GL}_{#1}} 
\newcommand{\symgroup}[1]{\mathcal{S}_{#1}} 
\newcommand{\qSchurAlg}[2]{S_q(#1,#2)} 
\newcommand{\hecke}[1]{\mathcal{H}_{#1}} 
\newcommand{\cuspidal}[1]{C_{#1}} 
\newcommand{\quantumfunctor}[1]{\mathcal{P}^{#1}_q}
\newcommand{\polyfunctor}[1]{\mathcal{P}^{#1}_1}
\newcommand{\source}[1]{\Gamma^{#1}_q \mathcal{V}}
\newcommand{\module}[1]{#1\mathrm{-mod}}
\newcommand{\cible}{\mathcal{V}}

\newcommand{\flagmod}[1]{\field[\mathcal{F}_{#1}]} 
\newcommand{\qSchur}[3]{S_q(#1 \times #2 , #3)}
\newcommand{\triv}[1]{\mathcal{I}_{#1}} 
\newcommand{\sign}[1]{\mathcal{E}_{#1}} 
\newcommand{\dividedmod}[1]{\dot{\Gamma}^{#1}} 
\newcommand{\symflag}[1]{\dot{S}^{#1}}
\newcommand{\Image}{\mathrm{Im}} 
\newcommand{\standard}[2]{V_{#1}^{\otimes #2}} 
\newcommand{\weightmod}[1]{V_{#1}} 

\newcommand{\extreso}[2]{C_{#1}(#2)}
\newcommand{\ribext}[2]{D_{#1}(#2)}
\newcommand{\Induction}[2]{\mathrm{Ind}_{#1}^{#2}} 
\newcommand{\Restriction}[2]{\mathrm{Res}_{#1}^{#2}} 
\newcommand{\Hom}[1]{\mathrm{Hom}_{#1}} 
\newcommand{\Endo}[1]{\mathrm{End}_{#1}}
\newcommand{\annihilator}[1]{\mathrm{ann}_{#1}} 
\newcommand{\Ext}[2]{\mathrm{Ext}_{#1}^{#2}}

\newcommand{\qdivided}[1]{\Gamma_q^{#1}}
\newcommand{\qexterior}[1]{\Lambda_q^{#1}}
\newcommand{\Schur}[1]{\mathbb{S}_{#1}}
\newcommand{\qtensor}[1]{I_q^{\otimes #1}}
\newcommand{\qsymmetric}[1]{S_q^{#1}}
\newcommand{\Weyl}[1]{\mathbb{W}_{#1}}
\newcommand{\Simple}[1]{\mathbb{L}_{#1}}

\newcommand{\Coproduct}[1]{\Delta^{#1}}

\newcommand{\divided}[1]{\Gamma^{#1}}
\newcommand{\exterior}[1]{\Lambda^{#1}}
\newcommand{\tensor}[1]{I^{\otimes #1}}
\newcommand{\symmetric}[1]{S^{#1}}
\newcommand{\Id}{\mathrm{Id}}
\newcommand{\Vect}{\mathrm{Vect}}
\newcommand{\row}[1]{\mathrm{row}(#1)}
\newcommand{\col}[1]{\mathrm{col}(#1)}
\newcommand{\twist}[1]{#1^{(1)_q}}
\newcommand{\twistr}[2]{#1^{(#2)_q}}
\newcommand{\qdual}[1]{(#1)^{*}}
\newcommand{\qdualwithoutparenthesis}[1]{#1^{*}}
\newcommand{\contradual}[1]{#1^{\tau}}
\newcommand{\diagonal}{\Delta}
\newcommand{\sumfunctor}{\Sigma}
\theoremstyle{definition}
    \newtheorem{definition}{Definition}[section]
\theoremstyle{plain}
    \newtheorem{lemma}[definition]{Lemma}
    \newtheorem*{lemma*}{Lemma}
    \newtheorem{corollary}[definition]{Corollary}
    \newtheorem{proposition}[definition]{Proposition}
    \newtheorem{theorem}[definition]{Theorem}
    \newtheorem{example}[definition]{Example}

\title{Cohomology of $\GL{d}(\ffield)$ in non-defining characteristic via the quantum schur algebra}
\author{Theo Deturck}
\date{May 2026}

\begin{document}

\maketitle

\begin{abstract}
    Let $G = \GL{d}(\ffield)$ be the general linear group over a field of cardinal $q$, and let $\field$ be a field of positive characteristic which does not divide $q(q-1)$. Building on the works of Cline, Parshall, and Scott, we show how to compute Ext-groups between $\field G$-modules using the quantum Schur algebra. The main novelty is our ability to compute these Ext-groups in higher degree than what was done before. More precisely, let $\ell$ be the order of $q$ in $\field$. In previous work, this method enabled the computation of the cohomology groups $H^*(\GL{d},M)$ in degree $*\leq \ell-1$. We show that for a lot of modules $M$, we can compute these cohomology groups in higher degree, with an example where we can compute until degree $3(\ell-1)$. We also show some new result on Ext-groups between modules over the quantum Schur algebra along the way. 
\end{abstract}


\section{Introduction}

Let $q$ be a prime power, $\GL{d} = \GL{d}(\ffield)$ and $\field$ be a field of positive characteristic which does not divide $q(q-1)$. In \cite{cline1999generic}, Cline, Parshall and Scott have proven that the algebra $\field \GL{d}$ admits a quotient algebra $\cuspidal{d}$, called the cuspidal algebra, which is Morita equivalent to the $q$-Schur algebra $\qSchurAlg{d}{d}$. 

The $q$-Schur algebra $\qSchurAlg{n}{d}$ is a quantum deformation of the classical Schur algebra (which is the case $q=1$), and its modules are exactly the polynomial representations of the quantum deformation of $\GL{n}(\field)$ of degree $d$. Moreover, the $q$-Schur algebra is a quasi-hereditary algebra, enabling the use of combinatorial methods to study $\Ext{}{}$-groups of modules over the $q$-Schur algebra, and hence also of modules over the cuspidal algebra. In particular, this implies many vanishing results for $\Ext{}{}$-groups. Furthermore, the block structure of the $q$-Schur algebra is well understood, giving rise to even more vanishing results.

Therefore, it is natural to use the Morita equivalence between the $q$-Schur algebra $\qSchurAlg{d}{d}$ and the cuspidal algebra $\cuspidal{d}$ to compute $\Ext{}{}$-groups of $\field \GL{d}$-modules, in particular between modules that are also modules over the cuspidal algebra. This approach led Cline, Parshall and Scott to show that low-degree cohomology groups of such modules can be computed in the category $\qSchurAlg{d}{d}$-mod. More precisely, if $\ell$ is the order of $q$ in $\field$,
\[
    H^*(\GL{d},N) \simeq \Ext{\qSchurAlg{d}{d}}{*}(\qexterior{d}(d), \alpha(N)) \quad \text{for } * \leq \ell-1
\]
where $\alpha$ is the Morita equivalence, and $\qexterior{d}(d)$ is the unique one-dimensional $\qSchurAlg{d}{d}$-module. In \cite{shalotenko2020ext}, Shalotenko generalized this result to the computation of other $\Ext{}{}$-groups of $\field \GL{d}$-modules in degree $\leq \ell - 1$.

Both of them use a projective resolution $P_*$ of certain $\cuspidal{d}$-modules such that $P_0,P_1,...,P_{\ell-2}$ are also projective as $\field \GL{d}$-modules. We will call such modules good projectives in this introduction. Dually, we can also define a notion of good injectives. In this paper, we describe explicitly a large family of good projectives, and also a large family of good injectives, which are in fact the projectives used by Cline, Parshall and Scott.

Our innovation is to compute $\Ext{\field \GL{d}}{}(M,N)$ using both a projective resolution of $M$ starting with good projectives, and an injective coresolution of $N$ starting with good injectives at the same time. This enables us to relate $\Ext{\field \GL{d}}{*}(M,N)$ and $\Ext{\qSchurAlg{d}{d}}{*}(\alpha(M),\alpha(N))$ in degree higher than $\ell-1$. For example :
\begin{equation}\label{eq : intro}
    H^*(\GL{d}, \dividedmod{d}) \simeq \Ext{\qSchurAlg{d}{d}}{*}(\qexterior{d}(d), \qdivided{d}(d))\quad \text{for } * \leq 3(\ell-1)
\end{equation}
where $\qdivided{d}(d)$ is the projective cover of a particular simple $\qSchurAlg{d}{d}$-module, and $\dividedmod{d}$ the corresponding $\cuspidal{d}$-module under the Morita equivalence.

To obtain this type of result, we study the maximal integer $n$ such that a given $\cuspidal{d}$-module $M$ admits a projective resolution $P_*$ of $M$ such that $P_0,...,P_{n-1}$ belong to our family of good projectives. We denote this integer (which can be $\infty$) by $\projbound{M}$. Dually, for good injective, we have an integer $\injbound{M}$. Then we have the following theorem, which improves \cite[Theorem 12.4]{cline1999generic} (knowing that $\projbound{\field} = \ell - 1$).

\begin{theorem}\label{thm : intro 1}
    Let $M,N$ be $\cuspidal{d}$-modules. Then the natural map
    \[
         \Ext{\qSchurAlg{d}{d}}{*}(\alpha(M),\alpha(N)) \to \Ext{\field \GL{d}}{*}(M,N)
    \]
    induced by the inverse equivalence of $\beta$ is an isomorphism in degree $* \leq \projbound{M} + \injbound{N}$, and is an injection in degree $*=\projbound{M} + \injbound{N}+1$.
\end{theorem}

We then study in detail the integers $\projbound{M}$ and $\projbound{N}$. For this purpose, we use a category of functors denoted by $\quantumfunctor{d}$, and called the category of quantum polynomial functors of degree $d$, which is equivalent to the category of $\qSchurAlg{d}{d}$-modules. This change of point of view enables us to use more categorical tools and in particular a method based on an adjunction between two functors, the sum-diagonal adjunction. This adjunction can be expressed on the level of $\Ext{}{}$-groups by the following proposition.

\begin{proposition}
    Let $F \in \quantumfunctor{d}$ and $G \in \quantumfunctor{d;2}$ be a two-variable quantum polynomial functor (see definition \ref{def : quantum functor with several variable}). Then
    \[
        \Ext{\quantumfunctor{d}}{*}(F(n),G(n,n)) \simeq \Ext{\quantumfunctor{d;2}}{*}(F(n+m),G(n,m)) 
    \]
    where $n$ and $m$ are variables.
\end{proposition}

This adjunction is particularly useful for tackling a number of problems in computing $\Ext{}{}$-group. It was used in \cite{touze2018connectedness} to study similar integers $\projbound{M}$ and $\injbound{N}$ but in the classical setting (for modules over the classical Schur algebra). In fact, most of the results of \cite{touze2018connectedness} carry over to our setting, and the only difficulty is to adapt the sum-diagonal adjunction, which we do. In the classical case, this adjunction can be defined using the fact that we can compose strict polynomial functors, it is not possible in general to do so for quantum polynomial functors (see \cite{buciumas2019quantum} where a way to define a composition is discussed). Hence more care is needed in the quantum case. In this paper, the sum and diagonal are introduced using the restriction and induction functors between Hecke algebras and parabolic subalgebras. Generalization to other type than type $A$ might be possible.

This enables us to give a homological characterization of the integers $\projbound{M}$ and $\projbound{N}$, in terms of non-vanishing of certain $\Ext{}{}$-groups (proposition \ref{prop : characterization of projective bound}). We use this to compute some of them, and in particular $\injbound{\dividedmod{d}}$ which (for $d \geq \ell$), is equal to $2(\ell-1)$. The isomorphism (\ref{eq : intro}) is then a direct consequence of the theorem \ref{thm : intro 1} and the fact that $\projbound{\field} = \ell-1$, which we also prove.

We end this article with a computation of $\Ext{\qSchurAlg{d}{d}}{*}(\qexterior{d}(d), \qdivided{d}(d))$ for all $d$ (and in fact, we only need $q$ to be a root of order $\ell$ in $\field$, and not that $q$ is a prime power) following the method of \cite{raicu2023stable}. This enables us to illustrate our method to compute $H^*(\GL{d}, \dividedmod{d})$ for all $d$ and all $* \leq 3(\ell-1)$, giving a lot of non-trivial cohomology groups. For example, we have the following computation of $\Ext{}{}$-groups giving some non-vanishing cohomology groups in degree higher than what was already known.

\begin{proposition}
    If $q$ has order $\ell > 2$ in $\field$, then $\dim H^*(\GL{3\ell},\dividedmod{3\ell}) = 1$ in degree $* = 3\ell-6$ and $3\ell-5$, and $H^*(\GL{3\ell},\dividedmod{3\ell}) = 0$ in every other degree $* \leq 3\ell-3$. Moreover, if $\field$ has characteristic $3$, then $H^*(\GL{3\ell},\dividedmod{3\ell}) \neq 0$ for $* = 3\ell-2$.
\end{proposition}

More generally, the computation of all $\dim H^*(\GL{d},\dividedmod{d})$ for $* \leq 3(\ell-1)$ can be found in the examples \ref{exam : coh 1} to \ref{exam : coh 4}.

We now briefly discuss the structure of this paper. It is divided into three parts. The first one presents the Morita equivalence itself, and does not contain any new result, except maybe the compatibility with the duality in proposition \ref{prop : morita}. We first describe briefly the $q$-Schur algebra and the Hecke algebra, and define the category of quantum polynomial functor. Then we talk about the cuspidal algebra, and define the Morita equivalence, which is formulated in terms of quantum polynomial functors instead of modules over the $q$-Schur algebra. This formulation seems more natural knowing \cite[Lemma 3.5c]{brundan2001quantum}. Then we present some explicit example of the correspondence between quantum polynomial functor and module over the cuspidal algebra.

The second part is devoted to the integer $\projbound{M}$ and $\injbound{N}$. We prove Theorem \ref{thm : intro 1} there. We adapt the sum diagonal adjunction used in the classical setting to our new setting, and then state some of the consequences of this adjunction, and in particular the characterization of $\projbound{M}$ in term of non-vanishing of some $\Ext{}{}$-group. We then compute it for several modules.

The third and last part is nearly independent. We compute in the category of quantum polynomial functor some $\Ext{}{}$-groups that are highly non-trivial. The method relies on the analysis of some quotient of the cobar complex of the quantum exterior algebra, giving resolutions of some important quantum polynomial functor. The complexes we obtain to compute $\Ext{}{}$-group are part of the reduced bar complex of the quantum divided power algebra, and we compute entirely its homology using simultaneously results on blocks of the $q$-Schur algebra, several short exact sequences of complexes and the shuffle product of the bar complex. As an application of the result of the second part, we compute entirely $H^*(\GL{d}, \dividedmod{d})$ for $* \leq 3\ell-3$, and obtain some non-vanishing result in degree $3\ell - 2$.

\section{The Morita equivalence}\label{sec : morita}

This section presents a classical method used to study the modular representation theory of finite general linear groups, which uses a functor from the category of modules over the quantum Schur algebra to the category of modules over $\field \GL{d}(\ffield)$, introduced by Cline, Parshall and Scott \cite{cline1999generic}. See also \cite{shalotenko2020ext} for concrete use of this functor. Our treatment closely follow \cite{brundan2001quantum} where all the missing proofs can be found. The purpose of this section is mainly to introduce notations and definitions that will be useful in the sequel. In section \ref{subsec : example morita}, we also make explicit a few examples of correspondence between modules over the $q$-Schur algebra and representations of  $\GL{d}$, which are implicit in the literature, and we prove a duality property (proposition \ref{prop : morita}) which does not seem to appear in the literature.

\subsection{The Hecke algebra and the quantum Schur algebra}\label{subsec : Hecke and Schur}

Let $\field$ be a field of positive characteristic $p>0$ and $q$ be a prime power. We denote $\GL{d} = \GL{d}(\ffield)$ and suppose that $p \nmid q(q-1)$. We also denote $\ell$ the multiplicative order of $q$ in $\field$.

\vspace{3mm}
In this subsection, we present two algebras that are central to the subject : the Hecke algebra $\hecke{d}$ and the $q$-Schur algebra. The Hecke algebra is a deformation of the algebra of the symmetric group $\field \symgroup{d}$ which naturally appears in the study of representations of $\field \GL{d}$ : it is the endomorphism algebra of the module of complete flags. The $q$-Schur algebra is a larger algebra which can be constructed as an endomorphism algebra of certain $\hecke{d}$-module, as it will be done in this subsection. This is also the endomorphism algebra of several $\field \GL{d}$-module, vas we will see in subsection \ref{subsec : cuspidal}. These two algebras will later be realized as endomorphism algebra of a certain representation of $\GL{d}$. We begin by introducing some notation.

\begin{definition}\label{def : composition}
    A composition of $d$ into $n$ parts is an $n$-tuple $\nu = (\nu_1,...,\nu_n)$ of non-negative integers whose sum $\nu_1 + \cdots + \nu_n$ is equal to $d$. We denote by $\composition{n}{d}$ the set of composition of $d$ into $n$ parts.
\end{definition}

\begin{definition}\label{def : symmetric group}
    Let $\symgroup{d}$ denote the symmetric group on $\{1,...,d\}$.
    \begin{itemize}
        \item For $\omega \in \symgroup{d}$, $\ell(\omega)$ is the number of inversions of $\omega$, and is called the length of $\omega$.
        \item We denote $s_i$ the transposition exchanging $i$ and $i+1$.
        \item Given a composition $\nu \in \composition{n}{d}$, we view $\symgroup{\nu} = \symgroup{\nu_1} \times \cdots \times \symgroup{\nu_n}$ as a subgroup of $\symgroup{d}$. This is the subgroup generated by the $s_i$ for $i \neq \nu_1,\nu_1+\nu_2,...,\nu_1+\nu_2+\cdots + \nu_{n-1}$.
        \item The set of distinguished left coset representatives of $\symgroup{\nu}$ is denoted $\leftcoset{\nu}$. This is the set of permutations $\sigma \in \symgroup{d}$ such that
        \[
            \sigma(i) < \sigma(i+1) \quad \text{for } i \not \in \{ \nu_1,\ \nu_1+\nu_2,\ ...,\ \nu_1 + \cdots + \nu_{n-1}\} \ .
        \]
        \item The set of distinguished right coset representatives of $\symgroup{\nu}$ is denoted $\rightcoset{\nu}$. It is the set of $\sigma \in \symgroup{d}$ such that $\sigma^{-1} \in \leftcoset{\nu}$.
        \item Given two compositions $\nu$ and $\mu$, the set of distinguished double coset representatives of $\symgroup{\nu} - \symgroup{\mu}$ is $\doublecosetrepr{\nu}{\mu} = \rightcoset{\nu} \cap \leftcoset{\mu}$.
    \end{itemize}
\end{definition}

\begin{definition}\label{def : Hecke algebra}
    The Iwahori-Hecke algebra $\hecke{d}$ associated to the symmetric group is the $\field$-algebra with basis $\{T_\omega\}_{\omega \in \symgroup{d}}$ and relations
    \[
        T_\omega T_{s_i} = \left \{ \begin{array}{cl}
            T_{\omega s_i} & \text{if } \ell(\omega s_i) > \ell(\omega), \\
            q T_{\omega s_i} + (q-1) T_{\omega} & \text{if } \ell(\omega s_i) < \ell(\omega),
        \end{array} \right .
    \]
    for all $\omega \in \symgroup{d}$ and all basic transposition $s_i$.
\end{definition}

We write $s_i = (i,i+1)$ for the basic transposition exchanging $i$ and $i+1$, and $T_i = T_{s_i}$. Each $\omega \in \symgroup{d}$ can be written as a product of basic transpositions $\omega = s_{i_1} s_{i_2} \cdots s_{i_{\ell(\omega)}}$, and we have $T_\omega = T_{i_1} T_{i_2} \cdots T_{i_{\ell(\omega)}}$. Therefore, the elements $T_i$ generate $\hecke{d}$. In fact, we can describe $\hecke{d}$ as the $\field$-algebra generated by $T_1,...,T_{d-1}$ with relations
\[
    T_iT_j = T_jT_i \  \text{ if } |i-j|>1, \quad T_iT_{i+1}T_i = T_{i+1}T_iT_{i+1}, \quad (T_i-q)(T_i+1) = 0.
\]
We will consider the following $\hecke{d}$-module, which we call standard module.
\begin{definition}\label{def : standard module}
    Let $V_n = \field^n$ with canonical basis $e_1,...,e_n$. We define $R : \standard{n}{2} \to \standard{n}{2}$ by
    \[
        R(e_i \otimes e_j) = \left \{ \begin{array}{cl}
            e_j \otimes e_i &  \text{if } i<j, \\
            - e_i \otimes e_i & \text{if } i=j, \\
            (q-1) e_i \otimes e_j + q e_j \otimes e_i & \text{if } i>j.
        \end{array} \right .
    \]
    Then $R$ defines an action of $\hecke{d}$ on $\standard{n}{d}$ by
    \[
        v \cdot T_i = (1^{i-1} \otimes R \otimes 1^{d-1-i})(v) \ .
    \]
\end{definition}
These standard modules decompose as a direct sum of cyclic $\hecke{d}$-modules, called weight modules :
\[
    \standard{n}{d} = \bigoplus_{\nu \in \composition{n}{d}} \weightmod{\nu} \quad \text{with } \ \weightmod{\nu} = e^{\otimes\nu} \hecke{d}
\]
(here $e^{\otimes\nu} = e_1^{\otimes \nu_1} \otimes e_2^{\otimes \nu_2} \otimes \cdots \otimes e_n^{\otimes \nu_n}$). These weight modules can be constructed by inducing characters from cerain subalgebras of $\hecke{d}$.

\begin{definition}\label{def : character of hecke}
    The Hecke algebra $\hecke{d}$ has exactly two linear characters, the trivial character $\triv{}$ and the sign character $\sign{}$, defined by 
    \[
        \triv{}(T_\omega) = q^{\ell(\omega)} \quad \text{and} \quad \sign{}(T_\omega) = (-1)^{\ell(\omega)} \ .
    \]
    Associated with the subgroup $\symgroup{\nu}$ of $\symgroup{d}$, we have the subalgebra $\hecke{\nu}$ of $\hecke{d}$ with basis $\{T_\omega\}_{\omega \in \symgroup{\nu}}$. The characters $\triv{}$ and $\sign{}$ of $\hecke{d}$ restrict to characters $\triv{\nu}$ and $\sign{\nu}$ of $\hecke{\nu}$.
\end{definition}

For later use, we define an involutive automorphism of the Hecke algebra that exchanges the two characters.

\begin{definition}\label{def : involution of hecke}
    We let $-^{\#}$ be the algebra automorphism of $\hecke{d}$ defined by $T_i^\# = (q-1) - T_i$.
\end{definition}

Note that $\standard{1}{d} \simeq \sign{}$ as an $\hecke{d}$-module. More generally, the weight modules are isomorphic to modules obtained by inducing the sign character from the subalgebra $\hecke{\nu}$.

\begin{lemma}\label{lemma : weight module as induced module}
    For $\nu \in \composition{n}{d}$, we have an isomorphism of $\hecke{d}$-module
    \[
        \weightmod{\nu} \simeq \Induction{\hecke{\nu}}{\hecke{d}}(\sign{\nu}) = \sign{\nu} \otimes_{\hecke{\nu}} \hecke{d} \ .
    \]
\end{lemma}

The space of homomorphisms between weight modules will play a central role in what follows, so we now introduces an explicit basis.

\begin{lemma}\label{lemma : hom between weight module}
    Let $\nu \in \composition{n}{d}$, $\mu \in \composition{m}{d}$. For $\sigma \in \doublecosetrepr{\mu}{\nu}$, we let
    \[
        \phi_{\nu,\mu}^{\sigma} : V_\nu \to V_\mu, \quad \phi_{\nu,\mu}^{\sigma}(e^{\otimes \nu} \cdot h) = e^{\otimes \mu} \cdot \left ( \sum_{\omega \in \symgroup{\mu} \sigma \symgroup{\nu} \cap \rightcoset{\mu}} T_\omega^{\#} \right ) h \ .
    \]
    Then $\{ \phi_{\nu,\mu}^{\sigma}\}_{\sigma \in \doublecosetrepr{\mu}{\nu}}$ is a basis of $\Hom{\hecke{d}}(\weightmod{\nu},\weightmod{\mu})$.
\end{lemma}
\begin{proof}
    See (1.2c) of \cite{brundan2001quantum}.
\end{proof}

\begin{definition}\label{def : quantum Schur algebra}
    For $n,m,d \geq 0$, we let
    \[
        \qSchur{n}{m}{d} = \Hom{\hecke{d}}(\standard{n}{d}, \standard{m}{d}).
    \]
    Then $\qSchurAlg{n}{d}$ is an algebra, called the $q$-Schur algebra.
\end{definition}

Note that Lemma \ref{lemma : hom between weight module} gives a basis of $\qSchur{n}{m}{d}$. In this article, we will see how to associate a $\field \GL{d}$-module to an $\qSchurAlg{n}{d}$-module. This module will not depend on the chosen $n$, assuming that $n \geq d$. Again assuming that $n \geq d$, the category of $\qSchurAlg{n}{d}$-module is equivalent to the category of quantum polynomial functor $\quantumfunctor{d}$, which we introduce below. Using the category of quantum polynomial functor at the place of the category of $\qSchurAlg{n}{d}$-module will be useful for several reasons. First, avoids to choose a value of $n$ (even if, in the computation, it will be easier to use such a choice). This will be useful when we will use tensor product of quantum polynomial functor, which changes the degree $d$. A second advantage is related to the sum-diagonal adjunction, which we will use in section \ref{sec : low degree ext}.

\begin{definition}\label{def : quantum polynomial functor}
    Let $\source{d}$ denote the category whose objects are the integers $n \geq 0$, and the homomorphisms are given by
    \[
        \Hom{\source{d}}(n,m) = \qSchur{n}{m}{d} \ .
    \]
    and let $\cible$ denote the category of finite dimensional $\field$-vector space. The category $\quantumfunctor{d}$ of (homogeneous) quantum polynomial functor of degree $d$ is the category of $\field$-linear functor $\source{d} \to \cible$. The category of all quantum polynomial functor is $\quantumfunctor{} = \bigoplus_{d \geq 0} \quantumfunctor{d}$.
\end{definition}
We will also use the term contravariant quantum polynomial functor of degree $d$ for $\field$-linear functor $\source{d}^{op} \to \cible$.

The category of quantum polynomial functors is a braided monoidal category (see \cite{hong2017quantum} for properties and examples of quantum polynomial functors).

\subsection{The cuspidal algebra}\label{subsec : cuspidal}

The cuspidal algebra $\cuspidal{d}$ is a quotient of a block of $\field \GL{d}$, the block containing the trivial module. It is deeply linked with the combinatorics of flags. This algebra is Morita equivalent to the $q$-Schur algebra, as we will see in this subsection. We will first construct some $\field \GL{d}$-modules which are also $\cuspidal{d}$-modules, and find a progenerator of the category of $\cuspidal{d}$-modules, with endomorphism algebra the $q$-Schur algebra. This module will therefore define a Morita equivalence, which we reformulate in terms of quantum polynomial functor.

\begin{definition}\label{def : flags}
    A flag in $\ffield^d$ is an increasing sequence of subspace
    \[
        0 = V_0 \subset V_1 \subset V_2 \subset \cdots \subset V_{n-1} \subset V_n = \ffield^d\ . 
    \]
\end{definition}

The group $\GL{d}$ acts naturally on the set of flags in $\ffield^d$, and the orbits of this action are parametrized by compositions. Given a flag $V_* = (V_0 \subset V_1 \subset \cdots \subset V_n)$ in $\ffield^d$, we associate a composition $\nu \in \composition{n}{d}$ by setting $\nu_i = \dim V_i - \dim V_{i-1}$. Then two flags are in the same orbit under the action of $\GL{d}$ if and only if they are associated with the same composition, and we denote by $\flag{\nu}$ this orbit. Considering $\flag{\nu}$ as a basis of a $\field$-vector space, we obtain a representation $\flagmod{\nu}$ of the group algebra $\field \GL{d}$. These modules enable us to give a geometric realization of the $q$-Schur algebra, as we will see now.

Let $\matrixwithfixedsum{n}{m}{d}$ denote the set of $n \times m$ matrices with positive integer entries whose sum is equal to $d$. Given $A \in \matrixwithfixedsum{n}{m}{d}$, we associate compositions $\row{A} \in \composition{n}{d}$ and $\col{A} \in \composition{m}{d}$ by
\[
    \row{A}_i = \sum_{j=1}^m a_{ij} \quad \text{and} \quad \col{A}_j = \sum_{i=1}^n a_{ij} \ .
\]
We denote by $\matrixwithfixedrowcol{\nu}{\mu}$ the set of $A \in \matrixwithfixedsum{n}{m}{d}$ such that $\row{A} = \nu$ and $\col{A} = \mu$. To such an $A$ we associate a diagram. For example, with $n=3$ and $m=4$, we obtain from $A$ the following diagram.
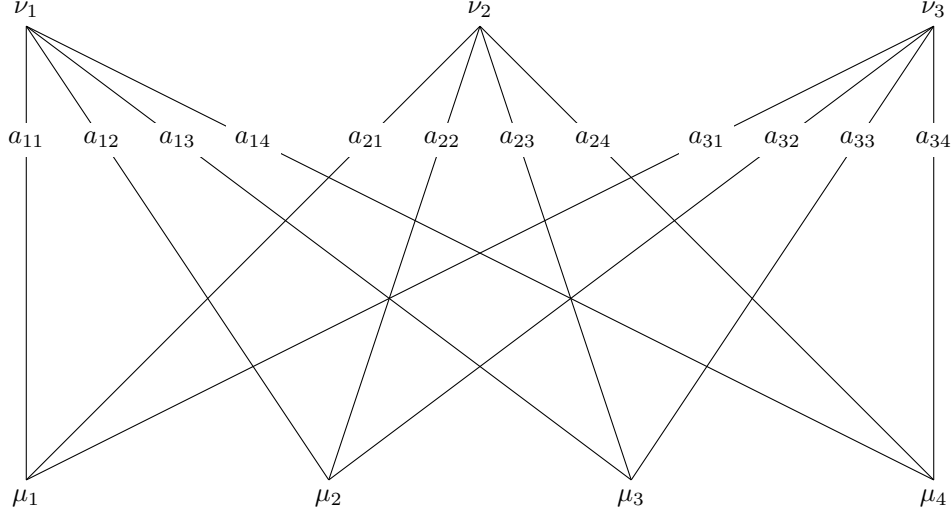
\begin{center}
    \begin{tikzpicture}[xscale = 2]
        \coordinate (source 1) at (0,6) ;
        \coordinate (source 2) at (3,6) ;
        \coordinate (source 3) at (6,6) ;

        \coordinate (target 1) at (0,0) ;
        \coordinate (target 2) at (2,0) ;
        \coordinate (target 3) at (4,0) ;
        \coordinate (target 4) at (6,0) ;

        \draw (source 1) -- (target 1) node[near start, fill = white] {$a_{11}$} ;
        \draw (source 1) -- (target 2) node[near start, fill = white] {$a_{12}$} ;
        \draw (source 1) -- (target 3) node[near start, fill = white] {$a_{13}$} ;
        \draw (source 1) -- (target 4) node[near start, fill = white] {$a_{14}$} ;
        \draw (source 2) -- (target 1) node[near start, fill = white] {$a_{21}$} ;
        \draw (source 2) -- (target 2) node[near start, fill = white] {$a_{22}$} ;
        \draw (source 2) -- (target 3) node[near start, fill = white] {$a_{23}$} ;
        \draw (source 2) -- (target 4) node[near start, fill = white] {$a_{24}$} ;
        \draw (source 3) -- (target 1) node[near start, fill = white] {$a_{31}$} ;
        \draw (source 3) -- (target 2) node[near start, fill = white] {$a_{32}$} ;
        \draw (source 3) -- (target 3) node[near start, fill = white] {$a_{33}$} ;
        \draw (source 3) -- (target 4) node[near start, fill = white] {$a_{34}$} ; 

        \node[above] at (source 1) {$\nu_1$};
        \node[above] at (source 2) {$\nu_2$};
        \node[above] at (source 3) {$\nu_3$};

        \node[below] at (target 1) {$\mu_1$};
        \node[below] at (target 2) {$\mu_2$};
        \node[below] at (target 3) {$\mu_3$};
        \node[below] at (target 4) {$\mu_4$};
    \end{tikzpicture}
\end{center}
This diagram represents a permutation of $\symgroup{d}$, that permute blocks of size $a_{ij}$. Formally, we obtain the permutation $\sigma_A$ defined by the equalities :
\[
    \sigma_A( \nu_1 + \cdots + \nu_{i-1} + a_{i,1} + \cdots + a_{i,j-1} + k) = \mu_1 + \cdots + \mu_{j-1} + a_{1,j} + \cdots + a_{i-1,j} + k \ \text{ for } 1 \leq k \leq a_{ij} \ . 
\]
Then $\sigma_A \in \doublecosetrepr{\mu}{\nu}$, and inversely every permutation in $\doublecosetrepr{\mu}{\nu}$ is of this form for a unique $A \in \matrixwithfixedrowcol{\nu}{\mu}$. This enables us to reparametrize our basis of lemma \ref{lemma : hom between weight module} by $\matrixwithfixedrowcol{\nu}{\mu}$, and hence our basis of $\qSchur{n}{m}{d}$ by $\matrixwithfixedsum{n}{m}{d}$.

\begin{lemma}\label{lemma : morphism between flag representation}
    Let $A \in \matrixwithfixedrowcol{\nu}{\mu}$. Set $\orbit{A}$ to be the set of pairs $(V_*,W_*)$ of flags in $\flag{\nu} \times \flag{\mu}$ such that for $i = 1,...,n$ and $j= 1,...,m$,
    \[
        \dim(V_i \cap W_j) - \dim(V_{i-1} \cap W_j) - \dim(V_i \cap W_{j-1}) + \dim(V_{i-1} \cap W_{j-1}) = a_{ij} \ .
    \]
    Define a linear map $e_A : \flagmod{\nu} \to \flagmod{\mu}$ by
    \[
        e_A(V_*) = \sum_{\substack{W_* \in \flag{\mu} \text{ such that} \\ (V_*,W_*) \in \orbit{A}}} W_* \ .
    \]
    Then $e_A$ is $\field \GL{d}$-linear, and the map
    \[
        \Hom{\hecke{d}}(\weightmod{\nu},\weightmod{\mu}) \to \Hom{\field \GL{d}}(\flagmod{\nu},\flagmod{\mu}), \quad \phi_{\nu,\mu}^{\sigma_A} \mapsto e_A
    \]
    is an isomorphism, and is compatible with all compositions.
\end{lemma}
\begin{proof}
    See Remark 1.4 in \cite{du1995note}.
\end{proof}

We note the following particular case.

\begin{corollary}\label{cor : endomorphism of complete flag}
    The algebra $\hecke{d}$ acts on the right on $\flagmod{(1^d)}$ by
    \[
        V_* \cdot T_i = \sum_{\substack{W_* \in \flag{(1^d)} \text{ such that} \\ \forall j \neq i, \ W_j = V_j \text{ and } W_i \neq V_i}} W_* \ .
    \]
    This action defines an algebra isomorphism
    \[
        \hecke{d} \simeq \Endo{\GL{d}}(\flagmod{(1^d)})^{op} \ .
    \]
\end{corollary}

Our progenerator is defined using this action. More precisely, since the actions of $\hecke{d}$ and $\field \GL{d}$ commute, the intersections of eigenspaces of the $T_i$ are $\field \GL{d}$-submodule of $\flagmod{(1^d)}$. Hence, we can define the following $\field \GL{d}$-module.

\begin{definition}\label{def : divided flag module}
    Let $\nu \in \composition{n}{d}$. Define the divided flag module $\dividedmod{\nu}$ by
    \[
        \dividedmod{\nu} = \{ x \in \flagmod{(1^d)} \ | \ \forall h \in \hecke{\nu}, \ x \cdot h = \sign{\nu}(h) x \} \simeq \Hom{\hecke{d}}(\weightmod{\nu}, \flagmod{(1^d)}) \ .
    \]
    We also define
    \[
        \dividedmod{n,d} = \bigoplus_{\nu \in \composition{n}{d}} \dividedmod{\nu} \ .
    \]
\end{definition}

\begin{theorem}\label{thm : hom between divided flag module}
    Let $\dividedmod{n,d} = \bigoplus_{\nu \in \composition{n}{d}} \dividedmod{\nu}$. Then
    \[
        \Hom{\field \GL{d}}(\dividedmod{n,d},\dividedmod{m,d}) \simeq \qSchur{m}{n}{d} \ .
    \]
    In fact, $\dividedmod{*,d} : n \mapsto \dividedmod{n,d}$ is a contravariant polynomial functor.
\end{theorem}
\begin{proof}
    See theorem 3.4c of \cite{brundan2001quantum} for a proof. We just construct the isomorphism for the reader's convenience. Let $\nu \in \composition{n}{d}$. Note that $\dividedmod{\nu} \simeq \Hom{\hecke{\nu}}(\sign{\nu},\flagmod{(1^d)}) \simeq \Hom{\hecke{d}}(\weightmod{\nu}, \flagmod{(1^d)})$ by adjunction between induction and restriction. Hence, for $\mu \in \composition{m}{d}$, we have a natural map
    \[
        \Hom{\hecke{d}}(\weightmod{\mu},\weightmod{\nu}) \to \Hom{\field \GL{d}}(\Hom{\hecke{d}}(\weightmod{\nu}, \flagmod{(1^d)}),\Hom{\hecke{d}}(\weightmod{\mu}, \flagmod{(1^d)})) \simeq \Hom{\field \GL{d}}(\dividedmod{\nu},\dividedmod{\mu}) \ .
    \]
    This map is an isomorphism, and the isomorphism of the theorem is just the direct sum of these maps : 
    \[
        \Hom{\field \GL{d}}(\dividedmod{n,d},\dividedmod{m,d}) \simeq \bigoplus_{\substack{\nu \in \composition{n}{d} \\ \mu \in \composition{m}{d}}} \Hom{\field \GL{d}}(\dividedmod{\nu},\dividedmod{\mu}) \simeq \bigoplus_{\substack{\nu \in \composition{n}{d} \\ \mu \in \composition{m}{d}}} \Hom{\hecke{d}}(\weightmod{\mu},\weightmod{\nu}) \simeq \qSchur{m}{n}{d} \ .
    \]
\end{proof}

The module $\dividedmod{n,d}$ are not projective in general as $\field \GL{d}$-module. But they are as module over the following algebra.

\begin{definition}\label{def : cuspidal algebra}
    Let $\annihilator{\field \GL{d}}(\flagmod{(1^d)})$ denote the two-sided ideal of $\field \GL{d}$ of elements $g$ such that $g \cdot \flagmod{(1^d)} = 0$. The cuspidal algebra is the quotient
    \[
        \cuspidal{d} = \field \GL{d} / \annihilator{\field \GL{d}}(\flagmod{(1^d)}) \simeq \Image(\field \GL{d} \to \Endo{\field}(\flagmod{(1^d)}) \ .
    \]
\end{definition}

From this definition, it is easy to see that all $\field \GL{d}$-modules we have defined so far are $\cuspidal{d}$-modules : they are all subquotients of $\flagmod{(1^d)}$ ($\flagmod{\nu}$ is a quotient of $\flagmod{(1^d)}$, the quotient map being obtained by forgetting some parts of the flags). In particular, $\dividedmod{n,d}$ is a $\cuspidal{d}$-module. In fact, under a certain condition on $n$, it is a progenerator.

\begin{proposition}\label{prop : progenerator of cuspidal algebra}
    If $n \geq d$, then $\dividedmod{n,d} = \bigoplus_{\nu \in \composition{n}{d}} \dividedmod{\nu}$ is a projective generator of $\module{\cuspidal{d}}$.
\end{proposition}
\begin{proof}
    See \cite[Theorem 3.4g]{brundan2001quantum}.
\end{proof}

Now, we have a progenerator of the category $\module{\cuspidal{d}}$ of cuspidal module, with endomorphism algebra $\qSchurAlg{n}{d}$. This gives a Morita equivalence, which we reformulate in terms of quantum polynomial functor.

\begin{corollary}\label{cor : morita 2}
    Consider the contravariant quantum polynomial functor
    \[
        \dividedmod{*,d} : \source{d}^{op} \to \module{\cuspidal{d}}, \quad \dividedmod{*,d}(n) = \dividedmod{n,d} \ .
    \]
    Let $\beta : \quantumfunctor{d} \to \module{\cuspidal{d}}$ denote the functor
    \[
        \beta(F) = \dividedmod{*,d} \otimes_{\source{d}} F \ .
    \]
    Then $\beta$ is an equivalence of category with inverse
    \[
        \alpha(M) : n \mapsto \Hom{\cuspidal{d}}(\dividedmod{n,d},M) \ .
    \]
\end{corollary}
By definition, $\dividedmod{*,d} \otimes_{\source{d}} F$ can be defined as the $\field \GL{d}$-module generated by the symbol $v \otimes_n w$, with $n \geq 0$, $v \in \dividedmod{n,d}$ and $w \in F(n)$, modulo relation $v \cdot \zeta \otimes_n w = v \otimes_m \zeta \cdot w$ for $\zeta \in \qSchur{n}{m}{d}$. But by theorem 4.7 of \cite{hong2017quantum}, we can use only one $n$ :
\[
    \beta(F) \simeq \dividedmod{n,d} \otimes_{\qSchurAlg{n}{d}} F(n) \quad \text{if } n \geq d.
\]


\subsection{Examples of correspondences between quantum polynomial functors and representations of $\GL{d}$}\label{subsec : example morita}

In this section, we give some examples of computations of $\beta(F)$, for some standard quantum polynomial functors $F$ : the tensor power $\qtensor{d}$, the quantum exterior power $\qexterior{\nu}$, the quantum divided power $\qdivided{\nu}$ and the quantum symmetric power $\qsymmetric{\nu}$. Finally, we will see that the functor $\beta$ transforms tensor products into Harish-Chandra inductions, and transforms a certain duality of the category $\quantumfunctor{d}$ into a duality induced by transpose of matrix in the category $\module{\cuspidal{d}}$.

\subsubsection{The tensor power functor and the complete flag module}

\begin{definition}\label{def : tensor power functor}
    The tensor power functor $\qtensor{d} \in \quantumfunctor{d}$ is given by $\qtensor{d}(n) = \standard{n}{d}$.
\end{definition}

If we let $\qtensor{1} = I_q$, we have $\qtensor{d} = \underbrace{I_q \otimes \cdots \otimes I_q}_{d \text{ times}}$ (using the monoidal structure of $\quantumfunctor{}$).

\begin{lemma}\label{lemma : morita tensor power}
    We have $\beta(\qtensor{d}) = \flagmod{(1^d)}$. Moreover, every $h \in \hecke{d}$ induces a natural transformation $\cdot h : \qtensor{d} \to \qtensor{d}$ given by the action of $h$ on $\standard{n}{d}$, and $\beta(\cdot h)$ is the action of $h$ on $\flagmod{(1^d)}$.
\end{lemma}
\begin{proof}
    Let $n \geq d$. 
    We have the following isomorphism of $\hecke{d}$-modules :
    \begin{align*}
        \beta(\qtensor{d}) & \simeq \Hom{\hecke{d}}(\standard{n}{d},\flagmod{(1^d)}) \otimes_{\qSchurAlg{n}{d}} \standard{n}{d} \\
        & \simeq \Hom{\hecke{d}}(\Hom{\qSchurAlg{n}{d}}(\standard{n}{d}, \standard{n}{d}),\flagmod{(1^d)}) \\
        & \simeq \Hom{\hecke{d}}(\hecke{d}, \flagmod{(1^d)}) \simeq \flagmod{(1^d)} \ .
    \end{align*}
    We justify these isomorphisms below.

    For the first isomorphism, since $\standard{n}{d} \simeq \bigoplus_{\nu \in \composition{n}{d}} \weightmod{\nu}$, 
    \[
        \dividedmod{n,d} \simeq \Hom{\hecke{d}}(\standard{n}{d},\flagmod{(1^d)}) \ .
    \]
    
    The second isomorphism follows from the projectivity of $\standard{n}{d}$ as a $\qSchurAlg{n}{d}$-module. To prove it, note that
    \[
        \standard{n}{d} \simeq \bigoplus_{\nu \in \composition{n}{d}}\weightmod{\nu} \simeq \Hom{\hecke{d}}(\weightmod{(1^d)}, \bigoplus_{\nu \in \composition{n}{d}} \weightmod{\nu}) \simeq \qSchurAlg{n}{d} \ \phi^{id}_{(1^d,0^{n-d}),(1^d,0^{n-d})} \ .
    \]
    The element $\phi^{id}_{(1^d,0^{n-d}),(1^d,0^{n-d})}$ is an idempotent, and hence $\standard{n}{d}$ is projective.

    The third isomorphism follows from Corollary 4.11 of \cite{hong2017quantum}, and the last one is evaluation at the unit element of $\hecke{d}$.
\end{proof}

\subsubsection{Quantum exterior power and flag module}

\begin{definition}\label{def : quantum exterior power}
    For $\nu \in \composition{m}{d}$, we define the quantum exterior power $\qexterior{\nu}$ by 
    \[
        \qexterior{\nu}(n) = \standard{n}{d} \otimes_{\hecke{\nu}} \triv{\nu} \simeq \standard{n}{d} \otimes_{\hecke{d}} \Induction{\hecke{\nu}}{\hecke{d}}(\triv{\nu}) \
    \]
    where $\triv{\nu}$ is the trivial character of $\hecke{\nu}$ viewed as a left module, see definition \ref{def : character of hecke}.
    We let $\qexterior{d} = \qexterior{(d)}$, with $(d)$ the composition of $d$ in one part. Then $\qexterior{\nu} \simeq \qexterior{\nu_1} \otimes \cdots \otimes \qexterior{\nu_m}$.
\end{definition}

With this definition, we obtain a map : 
\[
    \Hom{\hecke{d}}(\Induction{\hecke{\nu}}{\hecke{d}}(\triv{\nu}), \Induction{\hecke{\mu}}{\hecke{d}}(\triv{\mu})) \to \Hom{\quantumfunctor{d}}(\qexterior{\nu},\qexterior{\mu}) \ .
\]
This map is in fact an isomorphism (as a consequence of lemma \ref{lemma : morita exterior flag}).

\begin{lemma}\label{lemma : morita exterior flag}
    For $\nu \in \composition{m}{d}$, we have $\beta(\qexterior{\nu}) \simeq \flagmod{\nu}$. Moreover, for $A \in \matrixwithfixedrowcol{\nu}{\mu}$, $\beta(1 \otimes \phi^{\sigma_A}_{\nu,\mu}) : \flagmod{\nu} \to \flagmod{\mu}$ is given by the map $e_A$ of Lemma \ref{lemma : morphism between flag representation}.
\end{lemma}
\begin{proof}
    By \cite[Theorem 4.2a]{brundan2001quantum} and \cite[Remark 9.18c]{cline1999generic}, we have an isomorphism $\beta(\qexterior{\nu}) \simeq \flagmod{\nu}$. More precisely, as in the proof of Lemma \ref{lemma : morita tensor power}
    \[
        \beta(\qexterior{\nu}) \simeq \Hom{\hecke{d}}(\standard{n}{d}, \flagmod{(1^d)}) \otimes_{\qSchurAlg{n}{d}} (\standard{n}{d} \otimes_{\hecke{\nu}} \triv{\nu} ) \simeq \flagmod{(1^d)} \otimes_{\hecke{\nu}} \triv{\nu} \ .
    \]
    Then the natural quotient map $\flagmod{(1^d)} \twoheadrightarrow \flagmod{\nu}$ induces a quotient map $\flagmod{(1^d)} \otimes_{\hecke{\nu}} \triv{\nu} \twoheadrightarrow \flagmod{\nu}$, which is an isomorphism.
\end{proof}

\subsubsection{Quantum divided power and divided flag module}

\begin{definition}\label{def : quantum divided power}
    For $\nu \in \composition{m}{d}$, we define the quantum divided power $\qdivided{\nu}$ by
    \[
        \qdivided{\nu}(n) = \Hom{\hecke{d}}(\weightmod{\nu}, \standard{n}{d}) \simeq \{x \in \standard{n}{d} \ | \ \forall h \in \hecke{\nu}, \ x \cdot h = \sign{\nu}(h) x \}
    \]
    We have $\qdivided{\nu} \simeq \qdivided{\nu_1} \otimes \cdots \otimes \qdivided{\nu_m}$, where $\qdivided{d} = \qdivided{(d)}$ with $(d)$ the composition of $d$ in one part.
\end{definition}

The $\qdivided{\nu}$ form an important family of quantum polynomial functors because they are all projective, and generate $\quantumfunctor{}$. As with the quantum exterior power, the homomorphisms between different quantum divided power functors are given by part of the quantum Schur algebra. In fact

\begin{lemma}\label{lemma : morita quantum divided power}
    We have $\beta(\qdivided{\nu}) \simeq \dividedmod{\nu}$. Moreover, under this isomorphism,
    \[
        \beta(- \circ \phi^{\sigma_A}_{\mu,\nu}) = - \circ \phi^{\sigma_A}_{\mu,\nu} : \dividedmod{\nu} \to \dividedmod{\mu} \ .
    \]
\end{lemma}
\begin{proof}
    We have
    \begin{align*}
        \beta(\qdivided{\nu}) & \simeq \Hom{\hecke{d}}(\standard{n}{d}, \flagmod{(1^d)}) \otimes _{\qSchurAlg{n}{d}} \qdivided{\nu}(n) \\
        & \simeq \Hom{\hecke{d}}(\Hom{\qSchurAlg{n}{d}}(\qdivided{\nu}(n), \standard{n}{d}), \flagmod{(1^d)}) \\
        & \simeq \Hom{\hecke{d}}(\weightmod{\nu}, \flagmod{(1^d)}) \\
        & \simeq \dividedmod{\nu} \ .
    \end{align*}
    The second isomorphism follows from the projectivity of $\qdivided{\nu}(n)$ (see \cite[Theorem 4.7 and (4.0.30)]{hong2017quantum}) and the third is due to corollary 4.10 of \cite{hong2017quantum}. The result for homomorphisms follows from the fact that all the isomorphisms commute with the actions of the $q$-Schur algebras.
\end{proof}

\subsubsection{Quantum symmetric power and symmetric flag module}

\begin{definition}
    For $\nu \in \composition{m}{d}$, we define the quantum symmetric power functor $\qsymmetric{\nu} \in \quantumfunctor{d}$ by
    \[
        \qsymmetric{\nu} = \standard{n}{d} \otimes_{\hecke{\nu}} \sign{\nu} \ .
    \]
    We have $\qsymmetric{\nu} \simeq \qsymmetric{\nu_1} \otimes \cdots \otimes \qsymmetric{\nu_m}$. 
    We also define the symmetric flag module $\symflag{\nu} = \flagmod{(1^d)} \otimes_{\hecke{\nu}} \sign{\nu}$.
\end{definition}

The quantum symmetric powers are dual to the quantum divided power (see section \ref{subsub : tensor and duality}). They form a cogeneratrice family of injective quantum polynomial functor. As before, the homomorphisms between symmetric flag modules and those between quantum symmetric functors are given by parts of the $q$-Schur algebra.

\begin{lemma}\label{lemma : morita between symmetric}
    We have $\beta(\qsymmetric{\nu}) \simeq \symflag{\nu}$. Moreover we have an isomorphism
    \[
        \Hom{\quantumfunctor{d}}(\qsymmetric{\nu},\qsymmetric{\mu}) \simeq \Hom{\hecke{d}}(\weightmod{\nu},\weightmod{\mu}) \ .
    \]
    Under this isomorphism, $\phi^{\sigma}_{\nu,\mu} \in \Hom{\hecke{d}}(\weightmod{\nu},\weightmod{\mu})$ is sent to $n \mapsto \Id_{\standard{n}{d}} \otimes \phi^{\sigma}_{\nu,\mu} : \qsymmetric{\nu}(n) \to \qsymmetric{\mu}(n)$. The functor $\beta$ send this natural transformation to $\Id_{\flagmod{(1^d)}} \otimes \phi^{\sigma}_{\nu,\mu}$.
\end{lemma}
\begin{proof}
    The proof is similar to that of Lemma \ref{lemma : morita exterior flag}.
\end{proof}

\subsubsection{Tensor product and duality}\label{subsub : tensor and duality}

As we have said before, $\quantumfunctor{}$ is a braided monoidal category. Moreover, the category $\source{d}$ admits an involutive anti-automorphism, which is the identity on object and can be described on morphisms by
\[
    \tau : \qSchur{n}{m}{d} \to \qSchur{m}{n}{d}, \quad \phi_{\nu,\mu}^{\sigma} \mapsto \phi_{\mu,\nu}^{\sigma^{-1}}.
\]
This induces a duality in $\quantumfunctor{d}$ given by $F \mapsto \qdualwithoutparenthesis{F}$, where
\[
    \qdualwithoutparenthesis{F}(n) = \Hom{\field}(F(n),\field), \quad \text{and} \quad \qdualwithoutparenthesis{F}(\zeta) = - \circ \tau(\zeta) \ \text{ for } \zeta \in \qSchur{n}{m}{d}.
\]
An important property is that $\qdual{\qsymmetric{\nu}} = \qdivided{\nu}$ and vice-versa.

The category $\bigoplus_{d \geq 0} \module{\cuspidal{d}}$ is also a monoidal category, with tensor product given by Harish-Chandra induction, and has a duality induced by the transpose of matrix.

\begin{proposition}\label{prop : morita}
    The functors $\beta : \quantumfunctor{d} \to \module{\cuspidal{d}}$ combine to a functor $\beta : \quantumfunctor{} \to \bigoplus_{d \geq 0} \module{\cuspidal{d}}$ satisfying :
    \[
        \beta(F \otimes G) = \beta(F) \otimes \beta(G) \quad \text{and} \quad \beta(\qdualwithoutparenthesis{F}) = \beta(F)^\tau \ .
    \]
    Here the tensor product in $\bigoplus_{d \geq 0} \module{\cuspidal{d}}$ is defined using the Harish-Chandra induction and $M^\tau = \Hom{\field}(M,\field)$ with $\cuspidal{d}$-module structure defined by the involutive anti-automorphism of $\cuspidal{d}$ induced by the transposition.
\end{proposition}
\begin{proof}
    For the tensor product, see theorem 4.2a of \cite{brundan2001quantum}. For the involution, we use the fact that any quantum polynomial functor admits an injective coresolution by direct sums of functors $\qsymmetric{\nu}$. Let $F \in \quantumfunctor{d}$. Write the start of an injective coresolution of $F$ :
    \[
        0 \to F \to \bigoplus_{\nu} \qsymmetric{\nu} \to \bigoplus_{\mu} \qsymmetric{\mu} \ .
    \]
    Applying $\qdual{-}$, we obtain the start of an injective coresolution :
    \[
        \bigoplus_{\mu} \qdivided{\mu} \to \bigoplus_{\nu} \qdivided{\nu} \to \qdualwithoutparenthesis{F} \to 0
    \]
    Hence $\beta(\qdualwithoutparenthesis{F})$ is the cokernel of $\bigoplus_{\mu} \qdivided{\mu} \to \bigoplus_{\nu} \qdivided{\nu}$. Meanwhile, $\beta(F)$ is the kernel of $\bigoplus_{\nu} \qsymmetric{\nu} \to \bigoplus_{\mu} \qsymmetric{\mu}$, and hence $\contradual{\beta(F)}$ is the kernel of $\bigoplus_{\mu} \contradual{(\symflag{\mu})} \to \bigoplus_{\nu} \contradual{(\symflag{\nu})}$. Hence, we just need to show that there is a natural isomorphism $(\symflag{\nu})^\# \simeq \dividedmod{\nu}$. This is done as follows :
    \begin{align*}
        \qdual{\symflag{\nu}} & = \Hom{\field}(\symflag{\nu}, \field) \\
        & \simeq \Hom{\field}(\flagmod{d} \otimes_{\hecke{\nu}} \sign{\nu}, \field) \\
        & \simeq \Hom{\hecke{\nu}}(\sign{\nu}, \Hom{\field}(\flagmod{d}, \field)) \\
        & \simeq \Hom{\hecke{\nu}}(\sign{\nu}, \flagmod{d}) \\
        & = \dividedmod{\nu} \ .
    \end{align*}
    The isomorphisms follow from the adjunction between tensor and $\Hom{}$ and from an isomorphism $\contradual{\flagmod{(1^d)}} \simeq \flagmod{(1^d)}$ that we describe now. Let $e_1,...,e_d$ denote the canonical basis of $\ffield^d$, and $E_\bullet$ the complete flag with $E_k = \Vect(e_1,...,e_k)$. For any complete flag $V_\bullet$, choose $A \in \GL{d}$ such that $A \cdot E_\bullet = V_\bullet$. We use $A$ to define the dual flag $\psi(V_\bullet)$ of $V_\bullet$ with $k$-th component :
        \[
            \psi(V_\bullet)_k = \{v \in \ffield{q}^d \ | \ e_1^\tau A^\tau v = \cdots = e_k^\tau A^\tau v = 0 \}\ .
        \]
        where $\tau$ is the transpose. We can show that this does not depend on the choice of $A$, and that for any $g \in \GL{d}$,
        \[
            \psi(g \cdot V_\bullet) = (g^\tau)^{-1} \psi(V_\bullet) = (g^{-1})^\tau \psi(V_\bullet)
        \]
        Using this we can show that
        \[
            \Hom{\field}(\flagmod{d}, \field) \to \flagmod{d}, \quad f \mapsto \sum_{V_\bullet \in \flag{d}} f(V_\bullet) \psi(V_\bullet)
        \]
        is an isomorphism of $\field \GL{d}$-modules. 
        
        Naturality follows directly from the description the action of the quantum Schur algebras on $\symflag{\nu}$ and $\dividedmod{\nu}$.
\end{proof}

To compute further examples of $\beta(F)$, it is often useful to use resolution or coresolution by quantum exterior power, where it is often easier to write explicit formula for the different map using lemma \ref{lemma : morphism between flag representation}. Most of the quantum polynomial functors we will encounter in the rest of this article will have such resolution or coresolution.

\section{$\Ext{\field \GL{d}}{}$-groups in low degree}\label{sec : low degree ext}

In this section, we explain how to use the functor $\beta$ to compute $\Ext{}{}$-groups of $\field \GL{d}$-modules. In the next subsection, we explain which projective quantum polynomial functors $\qdivided{\nu}$ are send to projective $\field \GL{d}$-modules by $\beta$, and explain how these can be used to compute $\Ext{}{}$-groups in low degree. Then the rest of this section will be devoted to finding exactly until which degree of $\Ext{}{}$ we can compute in this manner.

\subsection{Projectivity of $\ell$-bounded divided flag module}\label{subsec : effect on projective}

If $F \in \quantumfunctor{d}$, then $\beta(F)$ is a $\cuspidal{d}$-module and hence a $\field \GL{d}$-module. We want to use $\beta$ to compute Ext-groups between $\field \GL{d}$-modules. The problem is that not all projective $\cuspidal{d}$-modules remains projective when seen as a $\field \GL{d}$-modules. The projective $\cuspidal{d}$-modules are the image by $\beta$ of projective quantum polynomial functors and, in particular, we have a large family $\beta(\qdivided{\nu}) = \dividedmod{\nu}$ of projective $\cuspidal{d}$-modules. In this subsection, we show that some of them remain projective as $\field \GL{d}$-modules, depending on the order $\ell$ of $q$ in $\field$, and deduce a method for computing $\Ext{}{}$-groups of $\field \GL{d}$-modules.

\begin{proposition}\label{prop : projectivity under the morita}
    Let $\nu \in \composition{n}{d}$. If all parts of $\nu$ are $< \ell$, then $\beta(\qdivided{\nu})$ is a projective $\field \GL{d}$-module.
\end{proposition}
\begin{proof}
    First, we note that $\flagmod{(1^d)}$ is projective as a $\field \GL{d}$-module, since $\flagmod{(1^d)}$ is the module induced from the trivial representation of the subgroup of upper triangular matrices in $\GL{d}$, whose cardinal is invertible in $\field$.
    
    As we have already seen,
    \[
        \beta(\qdivided{\nu}) \simeq \dividedmod{\nu} = \{ x \in \flagmod{(1^d)} \ | \ \forall h \in \hecke{\nu}, x \cdot h = \sign{\nu}(h) x \} \ .
    \]
    Let 
    \[
        y_\nu = \sum_{\omega \in \symgroup{\nu}} (-q^{-1})^{\ell(\omega)} T_\omega \in \hecke{\nu}.
    \]
    By proposition 2.2 of \cite{takeuchi1996group}, $\flagmod{(1^d)} y_\nu  \subset \dividedmod{\nu}$. Moreover, for $x \in \dividedmod{\nu}$
    \begin{align*}
        x y_\nu & = (-1)^{\ell(\omega_\nu)}  \sum_{\omega \in \symgroup{\nu}} (-q)^{\ell(\omega_\nu)- \ell(\omega)} (x \cdot T_\omega) \\
        & = (-1)^{\ell(\omega_\nu)}  \sum_{\omega \in \symgroup{\nu}} (-q)^{\ell(\omega_\nu)- \ell(\omega)} (-1)^{\ell(\omega)}  x \\
        & = \left ( \sum_{\omega \in \symgroup{\nu}} q^{\ell(\omega_\nu)- \ell(\omega)} \right ) x \\
        & = \left ( \sum_{\omega \in \symgroup{\nu}} q^{\ell(\omega)} \right ) x
         = \left (\prod_{i=1}^n \qfactorial{\nu_i} \right ) x
    \end{align*}
    where $\qfactorial{k} = \qinteger{1} \qinteger{2} \cdots \qinteger{k}$, $\qinteger{i} = 1 + q + q^2 + \cdots + q^{i-1}$. Since $q$ has order $\ell > 1$ in $\field^\times$, $\qinteger{k} \neq 0$ for $k<\ell$, and hence $\qfactorial{k} \neq 0$ for $k < \ell$. Under the assumption on the parts of $\nu$, $\left (\prod_{i=1}^n \qfactorial{\nu_i} \right ) \neq 0$. Hence, the map
    \[
        \pi : \flagmod{(1^d)} \to \dividedmod{\nu}, \quad x \mapsto \frac{1}{\prod_{i=1}^n \qfactorial{\nu_i}} x \cdot y_\nu 
    \]
    is a projection onto $\dividedmod{\nu}$, and is a homomorphism of $\field \GL{d}$-modules. Moreover, if $\iota : \dividedmod{\nu} \to \flagmod{(1^d)}$ denote the natural injection, $\pi \circ \iota$ is the identity of $\dividedmod{\nu}$. Since $\flagmod{(1^d)}$ is projective, this suffices to conclude that $\dividedmod{\nu}$ is also projective as a $\field \GL{d}$-module.
\end{proof}

A consequence of this is that if $F$ admits a projective resolution whose terms are direct sums of functors $\qdivided{\nu}$ with all parts of $\nu$ strictly smaller than the order of $q$ in $\field$, then applying $\beta$ yields a projective resolution of $\beta(F)$ as a $\field \GL{d}$-module. In this case, for all quantum polynomial functors $G$, $\Ext{\field \GL{d}}{}(\beta(F),\beta(G)) \simeq \Ext{\quantumfunctor{}}{}(F,G)$. This optimal situation is rare, but in most cases, we can compute $\Ext{}{}$-groups in low degree with this method, as we will see.

\begin{definition}\label{def : bounded resolution}
    A $\ell$-bounded projective is a quantum polynomial functor $\qdivided{\nu}$ with $\nu$ a composition with all parts $< \ell$. For $F \in \quantumfunctor{}$, we let $\projbound{F}$ denote the maximal integer $n$ (possibly $\infty$) such that $F$ admits a projective resolution $P_*$ where $P_0,...,P_{n-1}$ are direct sums of $\ell$-bounded projective. Similarly, an $\ell$-bounded injective is a quantum polynomial functor $\qsymmetric{\nu}$ with $\nu$ a composition with all parts $< \ell$, and we let $\injbound{G}$ denote the maximal integer $n$ such that $G$ admits an injective coresolution whose first $n$ terms are direct sums of $\ell$-bounded injectives.
\end{definition}

The integers $\projbound{F}$ and $\injbound{G}$ correspond under the Morita equivalence to the integers $\projbound{M}$ and $\injbound{N}$ of the introduction.

By duality, $\projbound{F} = \injbound{F^\#}$ and vice-versa. We will therefore focus on the computation of $\projbound{F}$ in the rest of this article. Their importance is illustrated by the following theorem, which is an improvement of \cite[Theorem 12.4]{cline1999generic} when combined with proposition \ref{prop : value of bound}.

\begin{theorem}\label{thm : compute ext}
    Let $F,G \in \quantumfunctor{d}$. Then the natural map
    \[
        \Ext{\quantumfunctor{}}{*}(F,G) \to \Ext{\field \GL{d}}{*}(\beta(F),\beta(G))
    \]
    is an isomorphism in degree $* \leq \projbound{F} + \injbound{G}$ and an injection in degree $* = \projbound{F} + \injbound{G}+1$
\end{theorem}
\begin{proof}
    Consider a projective resolution $P_*$ of $F$ with $P_0,...,P_{\projbound{F}-1}$ direct sums of $\ell$-bounded projectives, and an injective coresolution $J^*$ of $G$ with $J^0,...,J^{\injbound{G}-1}$ direct sums of $\ell$-bounded injectives. Let $K$ be the kernel of the map $P_{\projbound{F}-1} \to P_{\projbound{F}-2}$. For $*<\projbound{F}$, we have
    \[
        \Ext{\field \GL{d}}{*}(\beta(F),\beta(G)) \simeq \Ext{\quantumfunctor{}}{*}(F,G)
    \]
    by proposition \ref{prop : projectivity under the morita} and the fact that $\Hom{\cuspidal{d}}(M,N) = \Hom{\field \GL{d}}(M,N)$. For $\projbound{F} \leq * < \projbound{F} + \injbound{G}$, by dimension shifting \cite[Exercice 2.4.3]{weibel1994introduction},
    \[
        \Ext{\quantumfunctor{}}{*}(F,G) \simeq \Ext{\quantumfunctor{}}{*-\projbound{F}}(K,G) \quad \text{and} \quad \Ext{\field \GL{d}}{*}(\beta(F),\beta(G)) \simeq \Ext{\field \GL{d}}{*-\projbound{F}}(\beta(K),\beta(G)) \ .
    \]
    Now, just as $\ell$-bounded projectives remain projective, $\ell$-bounded injective remains injective when we apply $\beta$, and hence
    \[
        \Ext{\field \GL{d}}{*}(\beta(F),\beta(G)) \simeq \Ext{\field \GL{d}}{*-\projbound{F}}(\beta(K),\beta(G)) \simeq \Ext{\quantumfunctor{}}{*-\projbound{F}}(K,G) \simeq \Ext{\quantumfunctor{}}{*}(F,G) \ .
    \]
    For $* = \projbound{F} + \injbound{G}$, we let $I$ denote the image of $J^{\injbound{G}-1} \to J^{\injbound{G}}$ and we have by dimension shifting
    \[
        \Ext{\field \GL{d}}{\projbound{F} + \injbound{G}}(\beta(F),\beta(G)) \simeq \Hom{\field \GL{d}}(\beta(K),\beta(I)) \simeq \Hom{\quantumfunctor{}}(K,I) \simeq \Ext{\quantumfunctor{}}{\projbound{F} + \injbound{G}}(F,G) \ .
    \]
    Finally, for $* = \projbound{F} + \injbound{G} + 1$,
    \begin{align*}
        \Ext{\quantumfunctor{}}{\projbound{F} + \injbound{G} + 1}(F,G) \simeq \Ext{\quantumfunctor{}}{1}(K&,I) \simeq \Ext{\cuspidal{d}}{1}(\beta(K),\beta(I)) \\
        & \hookrightarrow \Ext{\field \GL{d}}{1}(\beta(K),\beta(I)) \simeq \Ext{\field \GL{d}}{\projbound{F} + \injbound{G}+1}(\beta(F),\beta(G)) \ .
    \end{align*}
        
\end{proof}

The remainder of this section is devoted to computing $\projbound{F}$, or at least to finding lower bounds for it, and is based on \cite{touze2018connectedness}, where all the proofs can be found in the classical case, and which follow mostly from a biadjunction that we describe in the next subsection.

\subsection{The sum diagonal adjunction}\label{subsec : sum-diagonal}

A useful tool in the study of strict polynomial functors is a certain adjunction between two functors, called the diagonal and the sum. It is used, for example, in the proof of theorem 2.13 of \cite{friedlander1997cohomology}. In this section, we adapt this adjunction to the quantum case, using some properties of the standard $\hecke{d}$-module $\standard{n}{d}$ and well-known properties of the Hecke algebra (the bi-adjunction between restriction and induction, and Mackey formula, see \cite[Chapter 9]{geck2000characters} for details). This enables us to adapt nearly all the results and proofs of \cite{touze2018connectedness} in the quantum case.

Our adjunction will be between the category $\source{d}$ and a larger category $\source{d;2}$, defined below.

\begin{definition}\label{def : quantum functor with several variable}
    Let $\source{d;2}$ denote the category whose objects are pairs of non-negative integers $(n_1,n_2)$ and homomorphisms
    \begin{align*}
        \Hom{\source{d;2}}((n_1,n_2),(m_1,m_2)) & = \bigoplus_{d_1+d_2=d} \qSchur{n_1}{m_1}{d_1} \otimes \qSchur{n_2}{m_2}{d_2} \\
        & \simeq \bigoplus_{d_1+d_2=d} \Hom{\hecke{d_1} \otimes \hecke{d_2}}(\standard{n_1}{d_1} \otimes \standard{n_2}{d_2}, \standard{m_1}{d_1} \otimes \standard{m_2}{d_2}) \ .
    \end{align*}
    We let $\quantumfunctor{d;2}$ denote the category of all linear functor $\source{d;2} \to \cible$. Such functors will be called quantum polynomial functors in two variables.
\end{definition}

We start by defining the diagonal functor $\diagonal : \source{d} \to \source{d;2}$. We want $\Delta(n) = (n,n)$, hence the name 'diagonal'. To define it on homomorphisms, note that if $d_1+d_2=d$, $\standard{n}{d_1} \otimes \standard{n}{d_2} = \Restriction{\hecke{d_1 \otimes \hecke{d_2}}}{\hecke{d}}(\standard{n}{d})$, where $\Restriction{}{}$ is the restriction functor. This motivates the following definition.

\begin{definition}\label{def : diagonal}
    We define the diagonal functor $\diagonal : \source{d} \to \source{d;2}$ by
    \[
        \diagonal(n) = (n,n) \quad \text{and} \quad \diagonal(f) = \bigoplus_{d_1 + d_2 = d} \Restriction{\hecke{d_1} \otimes \hecke{d_2}}{\hecke{d}}(f) \ .
    \]
\end{definition}

To define its adjoint, we need the following isomorphism:
\[
    \standard{n_1+n_2}{d} \simeq \bigoplus_{d_1+d_2 = d} \Induction{\hecke{d_1} \otimes \hecke{d_2}}{\hecke{d}}(\standard{n_1}{d_1} \otimes \standard{n_2}{d_2}) \ .
\]
This isomorphism can be seen as a consequence of the decomposition into weight module and the isomorphism of lemma \ref{lemma : weight module as induced module}. It can also be checked directly, as it is given by the formula $e^{\otimes \nu} \mapsto e^{\otimes (\nu_1,...,\nu_n)} \otimes e^{\otimes (\nu_{n+1},...,\nu_{n+m})}$. Since induction and restriction are adjoint, this motivates the following definition.

\begin{definition}\label{def : sum}
    We define the sum functor $\sumfunctor : \source{d;2} \to \source{d}$ by
    \[
        \sumfunctor(n,m) = n+m \quad \text{and} \quad \sumfunctor \left (\bigoplus_{d_1+d_2=d} f_{d_1,d_2} \right ) = \sum_{d_1+d_2=d} \Induction{\hecke{d_1} \otimes \hecke{d_2}}{\hecke{d}}(f_{d_1,d_2}) \ .
    \]
\end{definition}

With these definitions, the biadjunction sum-diagonal is a direct consequence of the biadjunction induction-restriction for Hecke algebra \cite[Proposition 9.1.7]{geck2000characters}.

\begin{proposition}\label{prop : adjunction sum diagonal}
    The functors $\diagonal : \source{d} \to \source{d;2}$ and $\sumfunctor : \source{d;2} \to \source{d}$ are bi-adjoint.
\end{proposition}

The utility of this adjunction is mostly contained in the following corollary.

\begin{corollary}\label{cor : adjunction sum diagonal for ext}
    Let $F \in \quantumfunctor{d}$ and $G \in \quantumfunctor{d;2}$. Then we have isomorphisms
    \[
        \Ext{\quantumfunctor{d}}{*}(F,G\circ \diagonal) \simeq \Ext{\quantumfunctor{d;2}}{*}(F \circ \sumfunctor, G) \quad \text{and} \quad \Ext{\quantumfunctor{d}}{*}(G\circ \diagonal, F) \simeq \Ext{\quantumfunctor{d;2}}{*}(G, F \circ \sumfunctor).
    \]
\end{corollary}
\begin{proof}
    See Lemma 1.5 in \cite[Introduction to functor homology]{franjou2003rational}.
\end{proof}

\begin{definition}\label{def : external tensor product}
    Given $F \in \quantumfunctor{d_1}$ and $G \in \quantumfunctor{d_2}$ with $d=d_1+d_2$, we can define $F \boxtimes G \in \quantumfunctor{d;2}$ by $(F \boxtimes G)(n_1,n_2) = F(n_1) \otimes G(n_2)$ and 
    \[
        (F \boxtimes G) \left (\bigoplus_{d'_1+d'_2=d} f_{(d'_1,d'_2)} \otimes f'_{(d'_1,d'_2)} \right ) = F(f_{(d_1,d_2)}) \otimes G(f'_{(d_1,d_2)})
    \]
    using the isomorphism 
    \[
        \Hom{\hecke{d_1} \otimes \hecke{d_2}}(\standard{n_1}{d_1} \otimes \standard{n_2}{d_2}, \standard{m_1}{d_1} \otimes \standard{m_2}{d_2}) \simeq \Hom{\hecke{d_1}}(\standard{n_1}{d_1}, \standard{m_1}{d_1} ) \otimes \Hom{\hecke{d_2}}(\standard{n_2}{d_2}, \standard{m_2}{d_2}) \ .
    \]
\end{definition}

The tensor product $F \otimes G$ can then be defined as $(F \boxtimes G) \circ \diagonal$. Hence, we have in general isomorphisms
\[
    \Ext{\quantumfunctor{d}}{*}(F,G_1 \otimes G_2) \simeq \Ext{\quantumfunctor{d;2}}{*}(F \circ \sumfunctor, G_1 \boxtimes G_2) \quad \text{and} \quad \Ext{\quantumfunctor{d}}{*}(G_1 \otimes G_2, F) \simeq \Ext{\quantumfunctor{d;2}}{*}(G_1 \boxtimes G_2, F \circ \sumfunctor).
\]
When the functor $F$ has good properties with respect to $\sumfunctor$, then this can be used to compute several $\Ext{}{}$-groups. In particular, if $F$ is additive or exponential. Since both will be used later, we introduce them now and state some of their main properties.

\subsubsection{Additive functors}

\begin{definition}\label{def : additive functor}
    A quantum polynomial functor $F \in \quantumfunctor{d}$ is called additive if 
    \[
        F \circ \sumfunctor \simeq F \boxtimes \field \oplus \field \boxtimes F
    \]
    where $\field$ denotes the constant functor equal to $\field$ in $\quantumfunctor{0}$.
\end{definition}

The consequence of the additivity of $F$ on $\Ext{}{}$-groups is a corollary of the following lemma.


\begin{lemma}\label{lemma : external tensor and hom}
    Let $F_1,G_1 \in \quantumfunctor{d_1}$, $F_2, G_2 \in \quantumfunctor{d_2}$ and $d=d_1+d_2$. Then 
    \[
        \Ext{\quantumfunctor{d;2}}{}(F_1 \boxtimes G_1, F_2 \boxtimes G_2) \simeq \Ext{\quantumfunctor{}}{}(F_1,F_2) \otimes \Ext{\quantumfunctor{}}{}(G_1,G_2) \ .
    \]
\end{lemma}
\begin{proof}
    As in \cite[Section 4]{hong2017quantum}, one can prove that evaluation at $(n,m)$ for $n,m \geq d$ defines an equivalence of categories $\quantumfunctor{d;2} \to \module{\qSchurAlg{(n,m)}{d}}$, where 
    \[
        \qSchurAlg{(n,m)}{d} = \bigoplus_{d_1+d_2=d} \qSchurAlg{n}{d_1} \otimes \qSchurAlg{m}{d_2} \ .
    \]
    In particular, evaluation at $(n,n)$ for $n \geq d$ is an equivalence of categories. Hence
    \begin{align*}
        \Ext{\quantumfunctor{d;2}}{}(F_1 \boxtimes G_1, F_2 \boxtimes G_2) 
        & \simeq \Ext{\qSchurAlg{(n,m)}{d}}{}(F_1(n) \otimes G_1(n), F_2(n) \otimes G_2(n)) \\
        & \simeq \Ext{\qSchurAlg{n}{d_1} \otimes \qSchurAlg{n}{d_2}}{}(F_1(n) \otimes G_1(n), F_2(n) \otimes G_2(n)) \\
        & \simeq \Ext{\qSchurAlg{n}{d_1}}{}(F_1(n), F_2(n)) \otimes \Ext{\qSchurAlg{n}{d_2}}{}(G_1(n),G_2(n)) \\
        & \simeq \Ext{\quantumfunctor{}}{}(F_1,F_2) \otimes \Ext{\quantumfunctor{}}{}(G_1,G_2) \ .
    \end{align*}
\end{proof}

\begin{corollary}\label{cor : ext additif tenseur}
    Suppose $F \in \quantumfunctor{d}$ is additive and $G_1 \in \quantumfunctor{d_1}$, $G_2 \in \quantumfunctor{d_2}$ with $d_1,d_2 > 0$. Then
    \[
        \Ext{\quantumfunctor{}}{*}(F,G_1 \otimes G_2) = 0 \quad \text{and} \quad  \Ext{\quantumfunctor{}}{*}(G_1 \otimes G_2,F) = 0 \ .
    \]
\end{corollary}
\begin{proof}
    We only prove the first equality, the second is proved similarly. By the adjunction sum-diagonal and additivity of $F$, we have
    \begin{align*}
        \Ext{\quantumfunctor{}}{}(F,G_1 \otimes G_2)
        & \simeq \Ext{\quantumfunctor{}}{}(F \circ \sumfunctor ,G_1 \boxtimes G_2) \\
        & \simeq \Ext{\quantumfunctor{}}{}(F \boxtimes \field \oplus \field \boxtimes F, G_1 \boxtimes G_2) \\
        & \simeq \Ext{\quantumfunctor{}}{}(F,G_1) \otimes \Ext{\quantumfunctor{}}{}(\field, G_2) \oplus \Ext{\quantumfunctor{}}{}(\field ,G_1) \otimes \Ext{\quantumfunctor{}}{}(F, G_2) \\
        & = 0
    \end{align*}
    since $\Ext{\quantumfunctor{}}{}(\field ,G_1)  = \Ext{\quantumfunctor{}}{}(\field, G_2) = 0$, as there are no extensions (or homomorphisms) between functors of different degrees.
\end{proof}

One can show that all additive quantum polynomial functors are direct sums of $I_q$ or of the identity functor twisted by the classical Frobenius twist a certain number of times, and then by the quantum Frobenius twist.

\subsubsection{Exponential functors}

\begin{definition}\label{def : exponential functor}
    Let $F = (F_d)_{d \geq 0}$ be a family of quantum polynomial functors with $F_d \in \quantumfunctor{d}$. We say that $F$ is an exponential quantum polynomial functor if
    \[
        \forall d \geq 0, \quad F_d \circ \sumfunctor \simeq \bigoplus_{d_1+d_2 = d} F_{d_1} \boxtimes F_{d_2} \ .
    \]
\end{definition}

We then have the following corollary to lemma \ref{lemma : external tensor and hom}.

\begin{corollary}\label{cor : ext exponential tenseur}
     Suppose $F$ is an exponential functor and $G_1 \in \quantumfunctor{d_1}$, $G_2 \in \quantumfunctor{d_2}$ with $d_1,d_2$. Then, for $d=d_1+d_2$
    \[
        \Ext{\quantumfunctor{}}{}(F_d,G_1 \otimes G_2) = \Ext{\quantumfunctor{}}{}(F_{d_1},G_1) \otimes \Ext{\quantumfunctor{}}{}(F_{d_2},G_2)
    \]
    and conversely
    \[
        \Ext{\quantumfunctor{}}{}(G_1 \otimes G_2,F_d) = \Ext{\quantumfunctor{}}{}(G_1,F_{d_1}) \otimes \Ext{\quantumfunctor{}}{}(G_2,F_{d_2}) \ .
    \]
\end{corollary}
\begin{proof}
    The proof is similar to that of Corollary \ref{cor : ext additif tenseur}.
\end{proof}

Example of exponential functor include: the quantum divided power $\qdivided{*} = (\qdivided{d})_{d \geq 0}$, the quantum exterior power $\qexterior{*}$, the quantum symmetric power $\qsymmetric{*}$ and exponential strict polynomial functors twisted by the quantum Frobenius twist. Moreover, tensor products of exponential functors, in the sense
\[
    (F \otimes G)_{d \geq 0} = \left ( \bigoplus_{d_1+d_2=d} F_{d_1} \otimes G_{d_2} \right )_{d \geq 0}
\]
are also exponential, as a corollary of the following lemma.

\begin{lemma}\label{lemma : tensor and sum}
    Let $F \in \quantumfunctor{d_1}$, $G \in \quantumfunctor{d_2}$ and $d=d_1+d_2$. Then we have an isomorphism in $\quantumfunctor{d;2}$
    \[
        (F \otimes G) \circ \sumfunctor \simeq (F \circ \sumfunctor) \otimes (G \circ \sumfunctor) \ .
    \]
\end{lemma}
\begin{proof}
    This is a formal consequence of Mackey's formula.
\end{proof}

\subsection{A homological characterization of the $\projbound{F}$}\label{subsec : homological}

In this subsection, we consider some consequences of the sum-diagonal adjunction, and establish a characterization of $\projbound{F}$ in terms of $\Ext{}{}$-groups. This will be a direct adaptation of \cite{touze2018connectedness} to the quantum case, with all proofs being entirely similar. We will make extensive use of the quantum Frobenius twist and the classical Frobenius twist.

\begin{definition}\label{def : strict polynomial functor}
    The category of strict polynomial functors is the category $\polyfunctor{}$, which is $\quantumfunctor{}$ with $q=1$. Every quantum polynomial functor that we have encountered admits a strict analogue. We denote them by removing the $q$ in the notation, so $\exterior{d}$ for $\qexterior{d}$, $\symmetric{d}$ for $\qsymmetric{d}$,...
\end{definition}

\begin{definition}\label{def : twisted functor}
    Let $F \in \polyfunctor{d}$. For $r \geq 1$, we denote by $\twistr{F}{r} \in \quantumfunctor{dp^{r-1}\ell}$ the quantum polynomial functor obtained by applying the classical Frobenius twist $r-1$ times to $F$, then the quantum Frobenius twist.
\end{definition}
See \cite[Section 1.3]{brundan2001quantum} and \cite{théo2025extgroupcategoryquantumpolynomial} in the case where $\ell$ is odd. In this article, we simply cite the fact about these functors when we need them. We first record the following theorem concerning cup products.

\begin{theorem}\label{thm : cup product}
    Let $F_1,G_1$ be quantum polynomial functors and $F_2,G_2$ be strict polynomial functors. The cup product induces a graded injective map
    \[
        \Ext{\quantumfunctor{}}{*}(F_1,G_1) \otimes \Ext{\quantumfunctor{}}{*}(\twist{F_2},\twist{G_2}) \hookrightarrow \Ext{\quantumfunctor{}}{*}(F_1 \otimes \twist{F_2}, G_1 \otimes \twist{G_2}) \ .
    \]
    Moreover, this map is an isomorphism in degree $k$ in the following situations :
    \begin{enumerate}
        \item when $\deg F_1 < \deg G_1$ and $k < \injbound{G_1}$;
        \item when $\deg F_1 > \deg G_1$ and $k < \projbound{F_1}$;
        \item when $\deg F_1 = \deg G_1$ and $k < \projbound{F_1} + \injbound{G_1}$.
    \end{enumerate}
\end{theorem}
\begin{proof}
    This is \cite[Theorem 3.6]{touze2018connectedness} and the entire proof adapts directly to the present setting. The only results that we have not yet discussed are the notion of bi-degree, and the decomposition
    \[
        \qsymmetric{\nu} \circ \sumfunctor \simeq \bigoplus_{\nu^1 + \nu^2 = \nu} \qsymmetric{\nu^1} \boxtimes \qsymmetric{\nu^2}
    \]
    where $\nu^1 + \nu^2$ is the sum of compositions, defined termwise. This decomposition is a direct consequence of the fact that $\qsymmetric{*}$ is exponential, and of lemma \ref{lemma : tensor and sum}. The notion of bi-degree refers to the fact that a functor $F \in \quantumfunctor{d;2}$ decomposes as a direct sum
    \[
        F = \bigoplus_{d_1 + d_2 = d} F^{(d_1,d_2)} \quad \text{with } \ F^{(d_1,d_2)}(n_1,n_2) = 1_{(d_1,d_2)}(n_1,n_2) \cdot F(n_1,n_2)
    \]
    where $1_{(d_1,d_2)}(n_1,n_2) \in \Hom{\hecke{d_1} \otimes \hecke{d_2}}(\standard{n_1}{d_1} \otimes \standard{n_2}{d_2}, \standard{n_1}{d_1} \otimes \standard{n_2}{d_2})$ is the identity. If only one of the $F^{(d_1,d_2)}$ is non-zero, we say that $F$ is of bi-degree $(d_1,d_2)$. Then by the fact that the $1_{(d_1,d_2)}(n_1,n_2)$ for $d_1+d_2 =d$ are orthogonal idempotents of the algebra $\bigoplus_{d_1+d_2=d} \qSchurAlg{n_1}{d_1} \otimes \qSchurAlg{n_2}{d_2}$, we deduce that two functors of different bi-degrees have no non-trivial extensions (and $\Hom{}$) between them. 
\end{proof}

We now reformulate the definition of $\projbound{F}$ in terms of projective covers of certain simple object of $\quantumfunctor{}$.

\begin{definition}\label{def : l restricted}
    The simple objects of $\quantumfunctor{d}$ are parametrized by partitions of $d$, which are sequences $\lambda = (\lambda_1,\lambda_2,...)$ of non-negative integers whose sum equals $d$ and such that $\lambda_i \geq \lambda_{i+1}$ for $i \geq 1$ (see \cite[Proposition 6.6]{hong2017quantum}). Such a partition is called $\ell$-restricted if for all $i$, $\lambda_i - \lambda_{i+1} < \ell$. A simple object $\Simple{\lambda}$ in $\quantumfunctor{d}$ associated with an $\ell$-restricted partition $\lambda$ is called an $\ell$-restricted simple.
\end{definition}

\begin{proposition}\label{prop : bound projective cover}
    The integer $\projbound{F}$ is the supremum of all $n \geq 0$ such that $F$ admits a projective resolution $P_*$ in which the first $n$ objects $P_0,...,P_{n-1}$ are direct sums of projective covers of $\ell$-restricted simple.
\end{proposition}
\begin{proof}
    Again, this is \cite[Proposition 4.1]{touze2018connectedness} and the proof adapts directly to the quantum case. The only new ingredients needed are the Steinberg tensor product theorem \cite[(1.3e)]{brundan2001quantum} and the Clausen-James theorem, which in the quantum case is the statement that a partition $\lambda$ is $\ell$-restricted if and only if $\Hom{\quantumfunctor{}}(\Simple{\lambda}, \qtensor{d}) \simeq \Hom{\quantumfunctor{}}(\qtensor{d}, \Simple{\lambda})$ is non-zero. To my knowledge, this theorem does not appear in the literature, but the proof can be adapted from \cite[Theorem B.10]{touze2018connectedness}.
\end{proof}

The next result gives a characterization of $\projbound{F}$ in terms of $\Ext{}{}$-groups, allowing one to compute it without always having to construct an explicit projective resolution of $F$.

\begin{proposition}\label{prop : characterization of projective bound}
    Let $F \in \quantumfunctor{d}$. Then $\projbound{F}$ is equal to the least integer $k$ such that 
    \[
        \Ext{\quantumfunctor{}}{k}(F, \Simple{\lambda}) \neq 0 \quad \text{for a simple } \Simple{\lambda} \text{ which is not } \ell \text{-restricted}.
    \]
    It is also equal to the least integer $k$ such that 
    \[
        \Ext{\quantumfunctor{}}{k}(F, \qtensor{r_0} \otimes \twist{(I^{\otimes r_1})} \otimes \twistr{(I^{\otimes r_2})}{2} \otimes \cdots \otimes \twistr{(I^{\otimes r_t})}{t} ) \neq 0 
    \]
    for some $r_0,...,r_t$ such that $d = r_0 + \ell r_1 + \ell p r_2 + \cdots + \ell p^{t-1} r_t$ and $r_0 \neq d$. In each case, if there is no such integer, $\projbound{F} = \infty$.
\end{proposition}
\begin{proof}
    This is a formal consequence of theorem \ref{thm : cup product} and proposition \ref{prop : bound projective cover}, as explained in the proof of \cite[Proposition 7.1]{touze2018connectedness}.
\end{proof}

\subsection{Properties of $\projbound{F}$}\label{subsec : property bound}

In this subsection, we give some tools for computing $\projbound{F}$, and compute it for several families of quantum polynomial functors: $\qexterior{d}$, $\qdivided{d}$, $\qsymmetric{d}$, and the ribbon quantum Schur functors $\Schur{\lambda/\mu}$, which are the skew-Schur functors defined as in \cite[Definition 6.7]{hashimoto1992quantum}, but associated with ribbon skew-shapes, see definition \ref{def : ribbon}.

\begin{proposition}\label{prop : property of the bound}
    Let $F,G \in \quantumfunctor{}$. Then
    \begin{enumerate}
        \item $\projbound{F \otimes G} = \min(\projbound{F},\projbound{G})$ ;
        \item $\projbound{F \oplus G} = \min(\projbound{F},\projbound{G})$ ;
        \item\label{enum : resolution for bound} If $F_*$ is a resolution of $F$, then $\projbound{F} \geq \min(i+\projbound{F_i} : i \geq 0)$.
    \end{enumerate}
\end{proposition}
\begin{proof}
    Except for claim \ref{enum : resolution for bound}, the proof of proposition 7.3 of \cite{touze2018connectedness} works the same here. So we only prove points \ref{enum : resolution for bound}. Let's start with a short exact sequence
    \[
        0 \to K \to E \to F \to 0.
    \]
    Take projective resolutions $P_*$ of $K$ and $Q_*$ of $E$ such that $P_i$ and $Q_j$ are direct sums of $\ell$-bounded projectives for $i < \projbound{K}$ and $j < \projbound{E}$. Then there exist a chain map $f : P_* \to Q_*$ which extends the map $K \to E$. By \cite[1.5.2]{weibel1994introduction}, the cone of this chain map $\mathrm{cone}(f)_* = P_{*-1} \oplus Q_*$ gives a projective resolution of $F$, and $\mathrm{cone}(f)_*$ is a direct sum of $\ell$-bounded projective for $* < \min(\projbound{K} + 1, \projbound{E})$. Now, given any resolution $F_*$ of $F$, we can cut it in short exact sequences, and a simple induction based on the above result  shows that for all $j \geq 0$, $\projbound{F} \geq \min(\projbound{F_i}+i : 0 \leq i < j)$, which is equivalent to claim \ref{enum : resolution for bound}.

\end{proof}

We now write an explicit resolution of $\qexterior{d}$, which we then use to compute $\projbound{\qexterior{d}}$.

\begin{proposition}\label{prop : projective resolution of quantum exterior}
    For any $d \geq 0$, $\qexterior{d}$ admits a projective resolution $P_*$ with 
    \[
        P_i = \bigoplus_{\nu \in \compositionsanszero{d-i}{d}} \qdivided{\nu}
    \]
    where $\compositionsanszero{n}{d}$ is the set of compositions of $d$ into $n$ \textbf{non-zero} parts.
\end{proposition}
\begin{proof}
    This projective resolution is the (homogeneous strand of degree $d$ of) the reduced cobar complex of quantum divided power. The fact that this gives a projective resolution of quantum exterior power is due to the fact that $\qexterior{*}$ is a Koszul algebra with Koszul dual $\qdivided{*}$, which is a consequence of \cite[Proposition 5.4]{théo2025extgroupcategoryquantumpolynomial} by applying $\qdual{-}$. 
\end{proof}

As a consequence, $\projbound{\qexterior{d}} \geq \ell - 1$ if $d \geq \ell$. To see that this is an equality, we will use the following proposition.

\begin{proposition}\label{prop : ext exterior twisted}
    Let $G \in \polyfunctor{r}$ be a strict polynomial functor of degree $r$ and $F \in \quantumfunctor{s}$. Set $d=r\ell + s$. Then
    \[
        \Ext{\quantumfunctor{}}{*}(\qexterior{d}, F \otimes \twist{G}) \simeq \Ext{\quantumfunctor{}}{*}(\qexterior{s},F) \otimes \Ext{\polyfunctor{}}{*-r(\ell-1)}(\exterior{r},G) \ .
    \]
\end{proposition}
\begin{proof}
    Since $\qexterior{*}$ is exponential,
    \[
        \Ext{\quantumfunctor{}}{*}(\qexterior{d}, F \otimes \twist{G}) \simeq \Ext{\quantumfunctor{}}{*}(\qexterior{s},F) \otimes \Ext{\quantumfunctor{}}{*}(\qexterior{r\ell}, \twist{G}) \ .
    \]
    Hence we need to show 
    \[
        \Ext{\quantumfunctor{}}{*}(\qexterior{r\ell}, \twist{G}) \simeq   \Ext{\polyfunctor{}}{*-r(\ell-1)}(\exterior{r},G) \ .
    \]
    We use the projective resolution of proposition \ref{prop : projective resolution of quantum exterior} and the projective resolution of $\exterior{r}$ given in the same way by the reduced cobar complex of $\divided{*}$, the classical divided power functor. Thus $\Ext{\polyfunctor{}}{*}(\exterior{r},G)$ is the homology of 
    \[
        \bigoplus_{\nu \in \compositionsanszero{r-*}{r}} \Hom{\polyfunctor{}}(\divided{\nu}, G)
    \]
    and $\Ext{\quantumfunctor{}}{*}(\qexterior{r\ell}, \twist{G})$ of
    \[
        \bigoplus_{\mu \in \compositionsanszero{r\ell-*}{r\ell}} \Hom{\quantumfunctor{}}(\qdivided{\mu}, \twist{G}) \ .
    \]
    By applying the duality $\qdual{-}$ to \cite[Lemma 2.10]{deturck2026quantum}, we have an isomorphism
    \[
        \Hom{\quantumfunctor{}}(\qdivided{\nu},\twist{G}) \simeq \left \{ \begin{array}{cl}
            \Hom{\polyfunctor{}}(\divided{\nu/\ell}, G) & \text{if } \ell \text{ divides } \nu, \\
           0  & \text{otherwise.}
        \end{array} \right .
    \]
    Hence
    \[
        \bigoplus_{\nu \in \compositionsanszero{r-*}{r}} \Hom{\polyfunctor{}}(\divided{\nu}, G) \simeq \bigoplus_{\nu \in \compositionsanszero{r-*}{r}} \Hom{\quantumfunctor{}}(\divided{\ell \nu}, \twist{G}) \simeq \bigoplus_{\mu \in \compositionsanszero{r-*}{r\ell}} \Hom{\quantumfunctor{}}(\divided{\mu}, \twist{G}) \ .
    \]
    The compatibility with the differentials is a consequence of the fact that the quantum Frobenius twist comes from a bialgebra morphism, see \cite[Section 7.2]{parshall1991quantum}. The result follows since $r-* = r\ell - (*+ r(\ell-1))$.
\end{proof}

\begin{proposition}\label{prop : value of bound}
    Let $d \geq 0$. If $d < \ell$, $\projbound{F} = \infty$ for all $F \in \quantumfunctor{d}$. If $d \geq \ell$,
    \begin{enumerate}
        \item $\projbound{\qdivided{d}} = 0$;
        \item $\projbound{\qexterior{d}} = \ell-1$;
        \item $\projbound{\qsymmetric{d}} = 2(\ell-1)$.
        \item Let $\alpha \in \compositionsanszero{n}{d}$. Consider the ribbon $\lambda/\mu$ associated to $\alpha$ (see definition \ref{def : ribbon}), and $\Schur{\lambda/\mu}$ the Schur functor associated to $\lambda/\mu$ (see \cite[Definition 6.7]{hashimoto1992quantum}, where $\Schur{\lambda/\mu}$ is denoted $L_{\tilde{\lambda}/\tilde{\mu}}$). Then 
    \[
        \projbound{\Schur{\lambda/\mu}} \geq \ell-2+k \ ,
    \]
    where $k$ is the minimal number of consecutive parts $\alpha_i,\alpha_{i+1},...,\alpha_{i+k-1}$ of $\alpha$ one needs to add to obtain an integer $\geq \ell$.
    \end{enumerate}
    
\end{proposition}
\begin{proof}
    Any $F \in \quantumfunctor{d}$ admits a projective resolution by direct sums of quantum divided power functors $\qdivided{\nu}$ with $\nu$ compositions of $d$. When $d < \ell$, all these functors are $\ell$-bounded projectives. Now suppose $d \geq \ell$. In the proof, we will use the notion of weight of a quantum polynomial functor. Interested readers may consult \cite[Section 2.3]{théo2025extgroupcategoryquantumpolynomial} and/or \cite[Section 4]{hong2017quantum} for details.
    \begin{enumerate}
        \item We have
        \[
            \Ext{\quantumfunctor{}}{0}(\qdivided{d}, \qtensor{d-\ell} \otimes \twist{(\tensor{1})}) = \Hom{\quantumfunctor{}}(\qdivided{d}, \qtensor{d-\ell} \otimes \twist{(\tensor{1})}) \simeq \field
        \]
        by a simple argument involving weights (see \cite[Proposition 2.11, Proposition 4.3]{théo2025extgroupcategoryquantumpolynomial} and \cite[Corollary 4.10]{hong2017quantum}). Hence $\projbound{\qdivided{d}} \leq 0$ by proposition \ref{prop : characterization of projective bound}, and hence we have the equality. 
        \item By proposition \ref{prop : projective resolution of quantum exterior}, we have $\projbound{\qexterior{d}} \geq \ell-1$. Moreover,
        \[
            \Ext{\quantumfunctor{}}{\ell-1}(\qexterior{d}, \qtensor{d-\ell} \otimes \twist{(\tensor{1})}) \simeq \Hom{\quantumfunctor{}}(\qexterior{d-\ell}, \qtensor{d-\ell}) \otimes \Hom{\polyfunctor{}}(\exterior{1}, \tensor{1}) \simeq \field
        \]
        by proposition \ref{prop : ext exterior twisted} and by \cite[Proposition 3.4 and Proposition 3.5]{théo2025extgroupcategoryquantumpolynomial}. Hence $\projbound{\qexterior{d}} \leq \ell-1$, establishing the equality.
        \item We use resolutions by quantum exterior powers which will be introduced in the next section (see definition \ref{def : resolution of ribbon by exterior} and proposition \ref{prop : resolution of ribbon by exterior}). The complex $\extreso{d-*}{(1^d)}$ is a resolution of $\qsymmetric{d}$. By proposition \ref{prop : property of the bound} and the computation of $\projbound{\qexterior{d}}$, we have
        \[
            \projbound{\extreso{d-*}{(1^d)}} = \left \{ \begin{array}{cl}
                \infty & \text{if } * < \ell - 1 \text{ or } * \geq d,\\
                \ell - 1 & \text{otherwise}.
            \end{array} \right .
        \]
        Hence, $\projbound{\qsymmetric{d}} \geq 2(\ell - 1)$. Moreover, since $\qsymmetric{*}$ is an exponential functor,
        \[
            \Ext{\quantumfunctor{}}{*}(\qsymmetric{d}, \qtensor{d-\ell} \otimes \twist{(\tensor{1})}) \simeq \Ext{\quantumfunctor{}}{*}(\qsymmetric{d-\ell}, \qtensor{d-\ell}) \otimes \Ext{\quantumfunctor{}}{*}(\qsymmetric{\ell}, \twist{(\tensor{1})}) \ .
        \]
        Using the projectivity of $\qtensor{d-\ell}$ and a weight argument as before, we can prove that $\Ext{\quantumfunctor{}}{*}(\qsymmetric{d-\ell}, \qtensor{d-\ell})$ concentrated in degree $0$ where it is of dimension $1$. To compute $\Ext{\quantumfunctor{}}{*}(\qsymmetric{\ell}, \twist{(\tensor{1})})$, we can use the resolution $\extreso{d-*}{(1^d)}$. Using the additivity of $\twist{(\tensor{1})}$ and corollary \ref{cor : ext additif tenseur}, we can prove
        \[
            \Ext{\quantumfunctor{}}{2(\ell-1)}(\qsymmetric{\ell}, \twist{(\tensor{1})}) \simeq \Ext{\quantumfunctor{}}{\ell - 1}(\qexterior{\ell}, \twist{(\tensor{1})}) \simeq \Ext{\polyfunctor{}}{0}(\exterior{1}, \tensor{1}) \simeq \field \ .
        \]
        Hence 
        \[
            \Ext{\quantumfunctor{}}{2(\ell-1)}(\qsymmetric{d}, \qtensor{d-\ell} \otimes \twist{(\tensor{1})}) \simeq \field
        \]
        showing the other inequality.
    \end{enumerate}
    For ribbon Schur functors, using the notations of the proposition, we can also use the resolution $\extreso{d-*}{\alpha}$ to show the inequality $\projbound{\Schur{\lambda/\mu}} \geq \ell - 2 + k$.
\end{proof}

\section{An example: cohomology of divided flag modules}\label{sec : ext symetric and exterior}

In this section, we study some quotients of the (reduced) cobar complex of quantum exterior power $\qexterior{*}$. We note that they provide resolutions of some skew-Schur functors : the ribbon Schur functors. We then use it to compute some Ext-groups, and in particular, the Ext-groups $\Ext{\quantumfunctor{}}{*}(\qsymmetric{d},\qexterior{d})$. These $\Ext{}{}$-groups are isomorphic, via the duality $\qdual{-}$ of quantum polynomial functors, with the Ext-groups $\Ext{\quantumfunctor{}}{*}(\qexterior{d},\qdivided{d})$. This closely follows Sections $5$ and $6$ of \cite{raicu2023stable}. We then state the consequences for the cohomology $H^*(\GL{d},\dividedmod{d})$ of divided flag modules.

By definition, the (reduced) cobar complex of the quantum exterior power $\qexterior{*}$ can be decomposed as a direct sum of homogeneous strands, which are of the form
\[
    \extreso{*}{1^d} = \bigoplus_{\gamma \in \compositionsanszero{*}{d}} \qexterior{\gamma}
\]
modulo some change of indices. The differentials $\partial : \extreso{k}{1^d} \to \extreso{k+1}{1^d}$ can be defined as follows. Let $\alpha \in \compositionsanszero{k}{d}$ and $\beta \in \compositionsanszero{k+1}{d}$. We define the parts $\partial_{\alpha,\beta}$ of the differentials $\partial$ going from $\qexterior{\alpha}$ to $\qexterior{\beta}$ as follows.
\begin{itemize}
    \item If there exist $1 \leq t \leq k$ such that $\alpha = (\beta_1,\beta_2,...,\beta_{t-1},\beta_t + \beta_{t+1}, \beta_{t+2},...,\beta_{k+1})$, then this $t$ is unique, and we let
    \[
        \partial_{\alpha,\beta} = (-1)^{t-1} ( \Id^{\otimes t-1} \otimes \Coproduct{\beta_t,\beta_{t+1}} \otimes \Id^{\otimes k-t}) \ ,
    \]
    where $\Coproduct{\beta_t,\beta_{t+1}} : \qexterior{\beta_t+\beta_{t+1}} \to \qexterior{\beta_t} \otimes \qexterior{\beta_{t+1}}$ denote the coproduct of the quantum exterior algebra, see \cite[Proposition 3.3]{théo2025extgroupcategoryquantumpolynomial}.
    \item Otherwise, we let $\partial_{\alpha,\beta} = 0$.
\end{itemize}
Extending each $\partial_{\alpha,\beta}$ by zero, we then have $\partial = \sum_{\alpha,\beta} \partial_{\alpha,\beta}$ where the sum runs over all $\alpha \in \compositionsanszero{k}{d}$ and $\beta \in \compositionsanszero{k+1}{d}$. To define the quotients of $\extreso{\bullet}{1^d}$, we introduce the following notions on compositions.

\begin{definition}\label{def : refinement}
    Let $\alpha \in \compositionsanszero{n}{d}$. A refinement of $\alpha$ is a composition $\gamma$ obtained from $\alpha$ by adding some consecutive parts. We set $\refinement{k}{\alpha}$ the set of refinements of $\alpha$ of length $k$.
\end{definition}

For example, if $\alpha = (3,2,1,5)$, then the refinement of $\alpha$ are $(3,2,1,5)$, $(5,1,5)$, $(3,3,5)$, $(3,2,6)$, $(6,5)$,$(5,6)$,$(3,8)$ and $(11)$.

Fix $\alpha \in \compositionsanszero{n}{d}$. Remark that if $\gamma \in \compositionsanszero{k}{d}$ is not a refinement of $\alpha$, then any composition $\nu \in \compositionsanszero{k+1}{d}$ such that
\[
    \gamma = (\nu_1,\nu_2,...,\nu_{t-1},\nu_t + \nu_{t+1},\nu_{t+2},\nu_{t+3},...,\nu_{k+1})
\]
for some $t$, is not a refinement of $\alpha$. Hence, the graded subspace of $\extreso{\bullet}{1^d}$ defined as the direct sum of the $\qexterior{\gamma}$ with $\gamma$ not a refinement of $\alpha$ is a sub-complexes. We will be interested by the quotient of $\extreso{\bullet}{1^d}$ by this sub-complexes, which is naturally identified with the following complex.

\begin{definition}\label{def : resolution of ribbon by exterior}
    Let $\alpha \in \compositionsanszero{n}{d}$. For $k \geq 1$, we let
    \[
        \extreso{k}{\alpha} = \bigoplus_{\gamma \in \refinement{k}{\alpha}} \qexterior{\gamma} \ .
    \]
    We define $\partial_{\alpha} : \extreso{k}{\alpha} \to \extreso{k+1}{\alpha}$ by $\partial_\alpha = \sum_{\gamma,\nu} \partial_{\gamma,\nu}$ where the sum run over all $\gamma \in \refinement{k}{\alpha}$ and $\nu \in \refinement{k+1}{\alpha}$. Then $\partial_\alpha$ defines a differential on $\bigoplus_{k \geq 1} \extreso{k}{\alpha}$. We denote by $\extreso{\bullet}{\alpha}$ the complex obtained.
\end{definition}

We consider several complex homomorphisms.
\begin{itemize}
    \item If $\beta$ is a refinement of $\alpha$, then any refinement of $\beta$ is also a refinement of $\alpha$. Thus, as before (in the particular case $\alpha = (1^d)$), we have a quotient map $\extreso{\bullet}{\alpha} \twoheadrightarrow \extreso{\bullet}{\beta}$.
    \item For $\beta = (\alpha_1,\alpha_2,...,\alpha_{t-1},\alpha_t + \alpha_{t+1}, \alpha_{t+2},...,\alpha_{n})$, we can identify the kernel of the complex map $\extreso{\bullet}{\alpha} \twoheadrightarrow \extreso{\bullet}{\beta}$ with $\extreso{\bullet}{\alpha_1,...,\alpha_t} \otimes \extreso{\bullet}{\alpha_{t+1},...,\alpha_n}$ : if $\gamma$ is a refinement of $\alpha$ but not a refinement of $\beta$, then this means that in the process of constructing $\gamma$ by adding consecutive parts of $\alpha$, we did not add $\alpha_t$ and $\alpha_{t+1}$. Thus we can write $\gamma$ as the concatenation $\gamma = \gamma' \cdot \gamma''$ with $\gamma'$ a refinement of $(\alpha_1,...,\alpha_t)$ and $\gamma''$ a refinement of $(\alpha_{t+1},...,\alpha_n)$. Hence
    \[
        \extreso{\bullet}{\alpha_1,...,\alpha_t} \otimes \extreso{\bullet}{\alpha_{t+1},...,\alpha_n} = \bigoplus_{\substack{\gamma' \in \refinement{}{\alpha_1,...,\alpha_t} \\ \gamma'' \in \refinement{}{\alpha_{t+1},...,\alpha_n}}} \qexterior{\gamma' \cdot \gamma''} \ ,
    \]
    which is a sub-complex of $\extreso{\bullet}{\alpha}$, is the kernel of the quotient map.
    \item Let $\beta = (\alpha_1,\alpha_2,...,\alpha_{t-1},\alpha_t + \alpha_{t+1}, \alpha_{t+2},...,\alpha_{n})$. We can define a complex map
    \[
        \psi_{\alpha,\beta} : \extreso{\bullet-1}{\beta} \to \extreso{\bullet}{\alpha_1,...,\alpha_t} \otimes \extreso{\bullet}{\alpha_{t+1},...,\alpha_n}
    \]
    using the coproduct. Let $\gamma \in \refinement{k}{\beta}$. Then there exist $i$ such that $\gamma_1 + \cdots + \gamma_{i-1} < \alpha_1 + \cdots + \alpha_t < \alpha_1 + \cdots + \alpha_t +  \alpha_{t+1} \leq \gamma_1 + \cdots + \gamma_{i-1} + \gamma_i$. Consider
    \[
        \gamma' = (\gamma_1,...,\gamma_{i-1}, \alpha_1 + \cdots + \alpha_t - \gamma_1 - \cdots - \gamma_{i-1})
    \]
    and
    \[
        \gamma'' = (\gamma_1 + \cdots + \gamma_i - \alpha_1 - \cdots - \alpha_t, \gamma_{i+1},...,\gamma_k).
    \]
    Then $\gamma' \in \refinement{i}{\alpha_1,...,\alpha_t}$ and $\gamma'' \in \refinement{k+1-i}{\alpha_{t+1},...,\alpha_n}$. Hence the concatenation $\gamma' \cdot \gamma'' \in \refinement{k+1}{\alpha}$. We let $\psi_{\alpha,\beta,\gamma} = (-1)^{k}\partial_{\gamma,\gamma' \cdot \gamma''} : \qexterior{\gamma} \to \qexterior{\gamma' \cdot \gamma''}$ and 
    \[
        \psi_{\alpha,\beta} = \sum_{\gamma \in \refinement{k}{\beta}} \psi_{\alpha,\beta,\gamma} \ \ ,
    \]
    extending $\psi_{\alpha,\beta,\gamma}$ by zero outside $\qexterior{\gamma}$ for each $\gamma$.
\end{itemize}
There is a simple relation linking these three complex maps, given in the following lemma.
\begin{lemma}\label{lemma : exact triangle}
    The sequence
    \[
        \extreso{\bullet-1}{\beta} \xrightarrow{\psi_{\alpha,\beta}} \extreso{\bullet}{\alpha_1,...,\alpha_t} \otimes \extreso{\bullet}{\alpha_{t+1},...,\alpha_n} \xrightarrow{\subseteq} \extreso{\bullet}{\alpha} \twoheadrightarrow \extreso{\bullet}{\beta}
    \]
    is an exact triangle.
\end{lemma}
\begin{proof}
    By definition, the cone of $\psi_{\alpha,\beta}$ is the complex $\cone{\bullet}{\psi_{\alpha,\beta}}$ with degree $k$ parts
    \[
        \cone{k}{\psi_{\alpha,\beta}} = \extreso{k}{\beta} \oplus \left ( \bigoplus_{i+j=k} \extreso{i}{\alpha_1,...,\alpha_t} \otimes \extreso{j}{\alpha_{t+1},...,\alpha_n} \right )
    \]
    and differentials $\partial_{cone} : \cone{k}{\psi_{\alpha,\beta}} \to \cone{k+1}{\psi_{\alpha,\beta}}$ 
    \[
        \partial_{cone}(a,b) = (\partial_{\beta}(a), \partial_{\otimes}(b) + (-1)^k\psi_{\alpha,\beta}(a)).
    \]
    Hence, $\cone{\bullet}{\psi_{\alpha,\beta}} = \extreso{\bullet}{\alpha}$.
\end{proof}
The first application of the lemma is the computation of the homology of the complexes $\extreso{\bullet}{\alpha}$. They are given by special skew-Schur functors, which we call ribbon functors.

\begin{definition}\label{def : ribbon}
    Let $\alpha \in \compositionsanszero{n}{d}$. Let $\lambda'$ be the partition of length $n$ and $\mu'$ be the partition of length $n-1$ given by
    \[
        \lambda'_i = d-n+i- \sum_{t=1}^{i-1} \alpha_t \quad \text{and} \quad \mu'_i = \lambda'_i - \alpha_i \ .
    \]
    Consider their conjugate $\lambda,\mu$. Then $\lambda/\mu$ is a skew shape called the ribbon associated with $\alpha$.
\end{definition}
We can represent skew-shape by diagram, see \cite[Definition 6.1]{hashimoto1992quantum}. The ribbon associated to $\alpha$ can be obtained by the following process. Consider $n$ columns $C_1,...,C_n$ of blocks, containing $\alpha_1,\alpha_2,...,\alpha_n$ blocks respectively. Glue the last block of $C_2$ on the right of the first block of $C_1$, then the last block of $C_3$ on the right of the first block of $C_2$, and continue like that until you have glued together every column. For example, with $\alpha = (3,2,1,5)$,
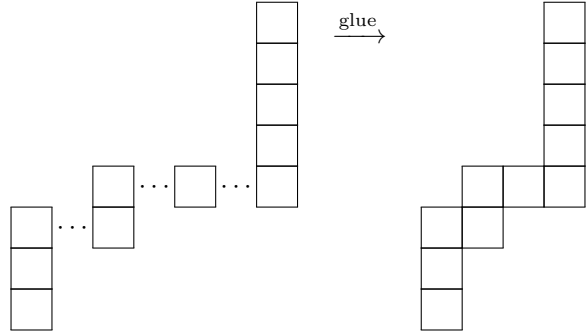
\[
    \begin{ytableau}
        \none & \none & \none & \none & \none & \none &\\
        \none & \none & \none & \none & \none & \none &\\
        \none & \none & \none & \none & \none & \none &\\
        \none & \none & \none & \none & \none & \none &\\
        \none & \none & & \none[\cdots] & & \none[\cdots] &\\
        & \none[\cdots] & & \none & \none & \none & \none\\
        & \none & \none & \none & \none & \none & \none\\
        & \none & \none & \none & \none & \none & \none\\
    \end{ytableau}
    \quad \xrightarrow{\mathrm{glue}} \quad 
    \begin{ytableau}
        \none  & \none & \none &\\
        \none  & \none & \none &\\
        \none  & \none & \none &\\
        \none  & \none & \none &\\
        \none  & & &\\
        & & \none & \none\\
        & \none & \none  & \none\\
        & \none & \none & \none\\
    \end{ytableau}
\]
Associated to a skew-shape $\lambda/\mu$, we have a quantum polynomial functor $\Schur{\lambda/\mu}$, which is called a skew-Schur functor. They are defined as the image of a natural transformation $\qexterior{\lambda'_1 - \mu'_1} \otimes \cdots \otimes \qexterior{\lambda'_h - \mu'_h} \to \qsymmetric{\lambda_1 - \mu_1} \otimes \cdots \otimes \qsymmetric{\lambda_n - \mu_n}$ (with $h = \lambda_1$). We will only use skew-Schur functor of the form $\lambda/\mu$ with $\lambda/\mu$ ribbon associated to some composition $\alpha$. As we explain in the following proposition, they can be defined as the cohomology of the complex $\extreso{\bullet}{\alpha}$.
\begin{proposition}\label{prop : resolution of ribbon by exterior}
    Let $\alpha \in \compositionsanszero{n}{d}$, and let $\lambda/\mu$ be the associated ribbon. Then
    \[
        H^*(\extreso{\bullet}{\alpha}) \simeq \left \{ \begin{array}{cl}
            \Schur{\lambda/\mu} & \text{if } * = n,  \\
            0 & \text{otherwise.} 
        \end{array} \right .
    \]
\end{proposition}
\begin{proof}
    By \cite[Theorem 6.19]{hashimoto1992quantum} (where our $\Schur{\lambda/\mu}$ is $L_{\lambda'/\mu'}$ with $\lambda',\mu'$ the conjugate partitions), the cokernel of the map
    \[
        \extreso{n-1}{\alpha} \xrightarrow{\partial} \extreso{n}{\alpha}
    \]
    is exactly $\Schur{\lambda/\mu}$. Since $\extreso{n+1}{\alpha} = 0$ (a refinement of $\alpha$ has less parts than $\alpha$), this proves that 
    \[
        H^n(\extreso{\bullet}{\alpha}) \simeq \Schur{\lambda/\mu} \ .
    \]
    Now, we prove by induction on $n$ that the cohomology is concentrated on degree $n$. For $n = 1$, $\extreso{\bullet}{\alpha}$ is concentrated in degree $1$, so there is nothing to prove.

    For $n > 1$, we have a short exact sequence,
    \[
        0 \to \extreso{\bullet}{\alpha_1 + \alpha_2,\alpha_3,...,\alpha_n} \to \extreso{\bullet}{\alpha} \to \extreso{\bullet}{\alpha_1} \otimes \extreso{\bullet}{\alpha_2,...,\alpha_n} \to 0.
    \]
    By induction hypothesis, the complex on the left has cohomology concentrated in degree $n-1$, and the complex on the right has cohomology concentrated in degree $n$. The long exact sequence associated with this short exact sequence then implies that $\extreso{\bullet}{\alpha}$ has cohomology concentrated in degree $n$.
\end{proof}
The quantum symmetric and quantum exterior power functors are ribbon Schur functors :
\[
    \Schur{(d)} = \qsymmetric{d} \quad \text{and} \quad \Schur{(1^d)} = \qexterior{d} \ .
\]
Hence, the proposition gives resolutions by exterior powers of these two functors :
\[
    H^d(\extreso{\bullet}{1^d}) = \qsymmetric{d} \quad \text{and} \quad H^1(\extreso{\bullet}{d}) = \qexterior{d} \ .
\]
We will use these resolutions to compute the Ext-groups $\Ext{\quantumfunctor{}}{*}(\qsymmetric{d},\qexterior{d})$. The following lemma shows that the non-projectivity of these resolutions is not important.
\begin{lemma}\label{lemma : ext between tensor of exterior and exterior}
    For any composition $\alpha \in \compositionsanszero{n}{d}$,
    \[
        \Ext{\quantumfunctor{}}{*}(\qexterior{\alpha},\qexterior{d})
        \simeq \left \{ \begin{array}{cl}
            \field & \text{if } *=0, \\
            0 & \text{otherwise.}
        \end{array} \right .
    \]
\end{lemma}
\begin{proof}
    The category $\quantumfunctor{}$ admits standard and costandard objects (see \cite[Chapter 9]{deng2008finite}) which are the Weyl functors $\Weyl{\mu}$ (the dual of the Schur module) and the Schur functors $\Schur{\mu}$). The functor $\qexterior{d}$ is both standard and costandard.
    \[ 
        \qexterior{d} = \Schur{(1^d)} = \Weyl{(1^d)}.
    \]
    Hence \cite[Proposition 9.31]{deng2008finite} shows that $\Ext{\quantumfunctor{}}{*}(\qexterior{d},\qexterior{d}) = 0$ for $*>0$. Using the exponentiality of $\qexterior{*}$, this implies $\Ext{\quantumfunctor{}}{*}(\qexterior{\alpha},\qexterior{d}) = 0$ for $*>0$. Finally, $\Ext{\quantumfunctor{}}{*}(\qexterior{\alpha},\qexterior{d}) = \Hom{\quantumfunctor{}}(\qexterior{\alpha},\qexterior{d})$ is given by Lemma \ref{lemma : morita exterior flag}, and is one-dimensional
\end{proof}

Since $\extreso{*}{\alpha}$ is a direct sum of tensor power of exterior functor, lemma \ref{lemma : ext between tensor of exterior and exterior} and proposition \ref{prop : resolution of ribbon by exterior} implies the following.

\begin{corollary}\label{cor : computation of ext between ribbon and exterior using the complexes}
    Let $\alpha \in \compositionsanszero{n}{d}$ and $\lambda/\mu$ be the corresponding ribbon. Let $\ribext{\bullet}{\alpha}$ denote the complex
    \[
        \ribext{*}{\alpha} = \Hom{\quantumfunctor{}}(\extreso{*}{\alpha},\qexterior{d}) \ .
    \]
    Then
    \[
        \Ext{\quantumfunctor{}}{*}(\Schur{\lambda/\mu},\qexterior{d}) \simeq H_{n-*}(\ribext{\bullet}{\alpha}) \ .
    \]
\end{corollary}
We will use this corollary to compute some Ext-groups of the form $\Ext{\quantumfunctor{}}{*}(\Schur{\lambda/\mu},\qexterior{d})$, but we will also use it in the other direction to compute some $H_{*}(\ribext{\bullet}{\alpha})$. More precisely, we will need the following computation.
\begin{proposition}\label{prop : ext between some hook and exterior}
    Let $a,b \geq 1$. If $\ell$ does not divide $a+b$ and does not divide $b$, then $\ribext{\bullet}{a,1^b}$ is acyclic.
\end{proposition}
\begin{proof}
    The shape associated to $(a,1^b)$ is the hook $(b+1,1^{a-1})$. We use the block theory of the classical and quantum Schur algebras. More precisely, we use the dual version of \cite[Theorem 5.37]{mathas1999iwahori}.
    \begin{lemma*}\label{lemma : block for Schur}
        Two (not skew-shape) Schur functors $\Schur{\lambda}$ and $\Schur{\mu}$ are in the same block if and only if the partitions $\lambda$ and $\mu$ have the same $\ell$-core.
    \end{lemma*}
    On one hand, $\qexterior{d} = \Schur{(1^d)}$, and the $\ell$-core of $(1^d)$ is $(1^{d'})$ where $d'$ is the rest in the division of $d$ by $\ell$. On the other hand, under our condition, the $\ell$-core of $(b+1,1^{a-1})$ is $(b'+1,1^{a'})$ where $b'$ is the rest in the division of $b$ by $\ell$ and $a'$ the rest in the division of $a-1$ by $\ell$. Since $b' \neq 0$, this show that $\qexterior{d}$ and $\Schur{(b+1,1^{a-1})}$ are not in the same block. Hence, 
    \[
        H_{b+1-*}(\ribext{\bullet}{a,1^b}) \simeq \Ext{\quantumfunctor{}}{*}(\Schur{(b+1,1^{a-1})}, \qexterior{d}) = 0.
    \]
\end{proof}

Let us describe concretely the complex $\ribext{\bullet}{\alpha}$. By definition,
\[
    \ribext{k}{\alpha} = \bigoplus_{\gamma \in \refinement{k}{\alpha}} \Hom{\quantumfunctor{}}(\qexterior{\gamma},\qexterior{d}) \ .
\]
The $\field$-vector space $\Hom{\quantumfunctor{}}(\qexterior{\gamma},\qexterior{d})$ is one dimensional, spanned by the product, that we will denote $[\gamma]$. Using the relation between product and coproduct of $\qexterior{*}$, we can prove that the differential is of the form
\[
    \partial[\gamma] = \sum_{t=1}^{k-1} (-1)^{t-1} \qbinom{\gamma_t + \gamma_{t+1}}{\gamma_t} [\gamma_1,...,\gamma_t + \gamma_{t+1},...,\gamma_{k}] \ .
\]
Our goal is to compute the homology of $\ribext{\bullet}{1^d}$. The main computation tools we will use are the complex morphisms induced by the precompositions by the complex morphism mentionned in lemma \ref{lemma : exact triangle}.
\begin{itemize}
    \item When $\beta$ is a refinement of $\alpha$, we have a quotient map $\extreso{\bullet}{\alpha} \twoheadrightarrow \extreso{\bullet}{\beta}$. This induces an inclusion $\iota : \ribext{\bullet}{\beta} \to \ribext{\bullet}{\alpha}$ (which sends $[\gamma]$ in $\ribext{\bullet}{\beta}$ to $[\gamma]$ in $\ribext{\bullet}{\alpha}$).
    \item The inclusion $\extreso{\bullet}{\alpha_1,...,\alpha_t} \otimes \extreso{\bullet}{\alpha_{t+1},...,\alpha_n} \subseteq \extreso{\bullet}{\alpha}$ induces a quotient map
    \[
        \pi : \ribext{\bullet}{\alpha} \twoheadrightarrow \ribext{\bullet}{\alpha_1,...,\alpha_t} \otimes \ribext{\bullet}{\alpha_{t+1},...,\alpha_n}
    \]
    given by
    \[
        [\gamma] \mapsto \left \{ \begin{array}{cl}
            [\gamma_1,...,\gamma_i] \otimes [\gamma_{i+1},...,\gamma_k] & \text{if } \gamma_1 + \cdots + \gamma_i = \alpha_1 + \cdots + \alpha_t, \\
            0 & \text{if the above equality is false for all } i. 
        \end{array} \right .
    \]
    \item The maps $\psi_{\alpha,\beta}$ induces complex morphisms
    \[
        \Psi : \ribext{\bullet}{\alpha_1,...,\alpha_t} \otimes \ribext{\bullet}{\alpha_{t+1},...,\alpha_n} \to \ribext{\bullet - 1}{\alpha_1,...,\alpha_t + \alpha_{t+1},...,\alpha_n}
    \]
    given by
    \[
        \Psi ([\gamma_1,...,\gamma_i] \otimes [\gamma_{i+1},....,\gamma_k]) = (-1)^{k-i-1} \qbinom{\gamma_i + \gamma_{i+1}}{\gamma_i} [\gamma_1,...,\gamma_i + \gamma_{i+1},...,\gamma_k] \ .
    \]
\end{itemize}

\begin{lemma}\label{lemma : exact triangle dual}
    Let $\alpha \in \compositionsanszero{n}{d}$, $\alpha' = (\alpha_1,...,\alpha_t)$, $\alpha'' = (\alpha_{t+1},...,\alpha_n)$ and $\beta = (\alpha_1,...,\alpha_{t} + \alpha_{t+1},...,\alpha_n)$. The sequence
    \[
        \ribext{\bullet+1}{\alpha'} \otimes \ribext{\bullet}{\alpha''}
        \xrightarrow{\Psi} \ribext{\bullet}{\beta}
        \xrightarrow{\iota} \ribext{\bullet}{\alpha} \xrightarrow{\pi} \ribext{\bullet}{\alpha'} \otimes \ribext{\bullet}{\alpha''}
    \]
    is an exact triangle. Hence, we have a long exact sequence
    \[
        \cdots \to (H(\alpha') \otimes H(\alpha''))_{*+1} \xrightarrow{\pm H_*(\Psi)} H_*(\beta) \xrightarrow{H_*(\iota)} H_*(\alpha) \xrightarrow{H_*(\pi)} (H(\alpha') \otimes H(\alpha''))_{*} \xrightarrow{\mp H_{*-1}(\Psi)} \cdots
    \]
    where we write $H_*(\alpha)$ for $H_*(\ribext{\bullet}{\alpha})$ to simplify notation.
\end{lemma}
\begin{proof}
    The proof is similar to lemma \ref{lemma : exact triangle}.
\end{proof}
Most of our computation will be done using the long exact sequence of lemma \ref{lemma : exact triangle dual}. We introduce one more useful tool for our computation : the shuffle product.
\begin{definition}\label{def : shuffle product}
    We define the following product on $\bigoplus_{d,k \geq 0} \ribext{k}{1^d}$ :
    \[
        [\gamma_1,...,\gamma_i] * [\gamma_{i+1},...,\gamma_k] = \sum_{\sigma \in \symgroup{}^{(i,k-i)}} (-1)^{\ell(\sigma)} [\gamma_{\sigma^{-1}(1)},...,\gamma_{\sigma^{-1}(d)}]
    \]
    where $\symgroup{}^{(i,k-i)}$ is the set of $(i,k-i)$-shuffle (see the section 4 of \cite{hashimoto1992quantum}, just before proposition 4.8). Here, $\ribext{\bullet}{1^0}$ denote the complex concentrated in degree $0$, where $\ribext{0}{1^0} = \field []$, $[]$ being a unit for the product.
\end{definition}
This product will be used to lift cycles of $\ribext{\bullet}{\alpha'} \otimes \ribext{\bullet}{\alpha''}$ to cycles of $\ribext{\bullet}{\alpha}$. This is justified by the following lemma.
\begin{lemma}\label{lemma : border of a shuffle}
    The shuffle product equips the complex $\bigoplus_{d,k \geq 0} \ribext{k}{1^d}$ with a structure of graded differential algebra. In other words, let $a \otimes b \in \ribext{i}{1^r} \otimes \ribext{j}{1^s}$. Then
    \[
        \partial(a * b) = \partial(a) * b + (-1)^i a * \partial(b) \ .    
    \]
\end{lemma}
\begin{proof}
    The complex $\bigoplus_{d \geq 0} \ribext{\bullet}{1^d}$ is isomorphic to the (reduced) bar complex of the q-divided power algebra $\qdivided{*}(1)$, shifted by one degree, which is commutative. The shuffle product is the usual shuffle product on the reduced bar complex, see section 4.2 of \cite{loday2013cyclic}.
\end{proof}


We can now state the main result.

\begin{theorem}\label{theorem : ext between symmetric and exterior powers}
    Suppose $\field$ is a field of characteristic $p>0$, and $d \geq 0$. Let
    \[
        A_d = \left \{(a_i)_{i \geq 0} \ | \ a_0 \equiv 0,1 \modulo \ell, \ a_i \equiv 0,1 \modulo p \text{ for } i \geq 1 \text{ and } a_0 + \sum_{i \geq 1} a_i\ell p^{i-1} = d \right \}
    \]
    For $a_0 \equiv 0,1 \modulo \ell$, let
    \[
        c_0(a_0) = \left \{ \begin{array}{cl}
            [\underbrace{\ell-1,1,\ell-1,1,...,\ell-1,1}_{\text{with } \frac{a_0}{\ell} \text{ times } \ell-1,1}] & \text{if } a_0 \equiv 0 \modulo \ell,\\
             
             [1,\underbrace{\ell-1,1,\ell-1,1,...,\ell-1,1}_{\text{with } \frac{a_0-1}{\ell} \text{ times } \ell-1,1}] & \text{if } a_0 \equiv 1 \modulo \ell.
        \end{array} \right .
    \]
    Hence $c_0(a_0) \in \ribext{2\frac{a_0}{\ell}}{1^{a_0}}$ if $a_0 \equiv 0 \modulo \ell$ and $c_0(a_0) \in \ribext{2\frac{a_0-1}{\ell}+1}{1^{a_0}}$ if $a_0 \equiv 1 \modulo \ell$.
    Similarly, for $i \geq 1$ and $a_i \equiv 0,1 \modulo p$, let
    \[
        c_i(a_i) = \left \{ \begin{array}{cl}
            [\underbrace{(p-1)\ell p^{i-1},\ell p^{i-1},...,(p-1)\ell p^{i-1},\ell p^{i-1}}_{\text{with } \frac{a_i}{p} \text{ times } (p-1)\ell p^{i-1},\ell p^{i-1}}] & \text{if } a_i \equiv 0 \modulo p,\\
             
             [\ell p^{i-1},\underbrace{(p-1)\ell p^{i-1},\ell p^{i-1},...,(p-1)\ell p^{i-1},\ell p^{i-1}}_{\text{with } \frac{a_i-1}{p} \text{ times } (p-1)\ell p^{i-1},\ell p^{i-1}}] & \text{if } a_i \equiv 1 \modulo p.
        \end{array} \right .
    \]
    Hence $c_i(a_i) \in \ribext{2\frac{a_i}{p}}{1^{\ell p^{i-1}a_i}}$ if $a_i \equiv 0 \modulo p$ and $c_i(a_i) \in \ribext{2\frac{a_i-1}{\ell}+1}{1^{\ell p^{i-1} a_i}}$ if $a_i \equiv 1 \modulo p$. Finally, for $a = (a_i)_{i \geq 0} \in A_d$ let
    \[
        c(a) = c_0(a_0) * c_1(a_1) * \cdots * c_r(a_r).
    \]
    with $r \geq 0$ big enough so that $a_i = 0$ for $i>r$. Then
    \[
        H_\bullet(1^d) = \bigoplus_{a \in A_d} \field \overline{c(a)}\ .
    \]
    where $\overline{c(a)}$ is the class of $c(a)$ in $H_\bullet(1^d)$.
\end{theorem}
Hence, we have the following Ext-groups computations.
\begin{corollary}\label{cor : ext between symmetric and exterior}
    Let $d \geq 0$. Then the dimension of $\Ext{\quantumfunctor{}}{k}(\qsymmetric{d},\qexterior{d})$ is equal to the number of $(a_0,a_1,a_2,...) \in A_d$ such that, letting $a_0 = b_0 \ell + c_0$, $a_i = b_i p + c_i$ (for $i \geq 1$) with $c_i \in \{0,1\}$ for $i \geq 0$, we have
    \[
        d-k = \sum_{i \geq 0} 2 b_i + c_i \ .
    \]
\end{corollary}
We will do the proof of theorem \ref{theorem : ext between symmetric and exterior powers} by induction, but we will first need several lemma.
\begin{lemma}\label{lemma : the cycle}
    Let $a \in A_d$. Then $c(a)$ is a cycle.
\end{lemma}
\begin{proof}
    By lemma \ref{lemma : border of a shuffle}, we only need to prove that the $c_i(a_i)$ are cycle. This is a consequence of Lucas's theorem for $q$-binomial coefficient which can be stated as follow. Let $n,k \geq 0$. In $\field$, consider the decompositions
    \[  
        n = n_0 + \ell n_1 + \ell p n_2 + \cdots + \ell p^{r-1} n_r, \quad 0 \leq n_0 < \ell \ \text{ and } \ 0 \leq n_i < p \ \text{ for } i \geq 1 \ ,
    \]
    \[  
        k = k_0 + \ell k_1 + \ell p k_2 + \cdots + \ell p^{r-1} k_r, \quad 0 \leq k_0 < \ell \ \text{ and } \ 0 \leq k_i < p \ \text{ for } i \geq 1 \ .
    \]
    Then
    \[
        \qbinom{n}{k} = \qbinom{n_0}{k_0} \binom{n_1}{k_1} \binom{n_2}{k_2} \cdots \binom{n_r}{k_r}
    \]
    where only the first binomial coefficient in the product is a $q$-binomial coefficient, the rest being classical binomial coefficient (see proposition 2.2 of \cite{desarmenien1982analogue}). In particular, $\qbinom{n}{k} = 0$ if and only if $n_i < k_i$ for at least one $i$. Consider first $c_0(a_0)$. Computing $\partial(c_0(a_0))$, we have a sum with all coefficients multiple of $\qbinom{\ell}{1} = \qbinom{\ell}{\ell-1}$. This coefficient is equal to $0$ by Lucas's theorem. For $c_i(a_i)$, this is the same things, but with all coefficients multiple of
    \[
        \qbinom{p^i \ell}{p^{i-1} \ell} = \qbinom{p^i \ell}{(p-1)p^{i-1} \ell} = 0
    \]
    again by Lucas's theorem.
\end{proof}
Another consequences of Lucas's theorem is the following.
\begin{lemma}\label{lemma : going from classical to quantum}
    Let $\underline{\ribext{\bullet}{\alpha}}$ denote the complex defined as $\ribext{\bullet}{\alpha}$ but with $q=1$ (if it was not already the case). Then the following map is an isomorphism of complex :
    \[
        \underline{\ribext{\bullet}{\alpha}} \to \ribext{\bullet}{\ell \alpha}, \quad [\gamma]\mapsto [\ell \gamma].
    \]
\end{lemma}
\begin{proof}
    By Lucas's theorem, the coefficients $\binom{\gamma_t + \gamma_{t+1}}{\gamma_t}$ and $\qbinom{\ell \gamma_t + \ell \gamma_{t+1}}{\ell \gamma_t}$ are equal, proving that the map is a morphism of complex, which is clearly bijective.
\end{proof}
We need one last lemma, which will be useful in a crucial step of the proof of the theorem.
\begin{lemma}\label{lemma : going from classical to quantum 2}
    Let $d = d' \ell$. Then the inclusion $\ribext{\bullet}{\ell^{d'}} \xrightarrow{\iota} \ribext{\bullet}{\ell,1^{d-\ell}}$ is a quasi-isomorphism.
\end{lemma}
\begin{proof}
    Consider the long exact sequence of lemma \ref{lemma : exact triangle dual} with $\alpha' = (\ell,1)$ and $\alpha'' = (\ell-1,\ell^{d'-2})$. By proposition \ref{prop : ext between some hook and exterior}, $H_*(\alpha') = 0$, showing that the inclusion
    \[
        \ribext{\bullet}{\ell^{d'}} \hookrightarrow \ribext{\bullet}{\ell,1,\ell-1,\ell^{d'-2}}
    \]
    is a quasi-isomorphism. We continue like that with $\alpha' = (\ell,1^i)$, $\alpha'' = (\ell-i, \ell^{d'-2})$ for $i=2,3,...,\ell-1$. In each case, the inclusion are quasi-isomorphism and hence the inclusion
    \[
        \ribext{\bullet}{\ell^{d'}} \hookrightarrow \ribext{\bullet}{\ell,1^{\ell},\ell^{d'-2}}
    \]
    is a quasi-isomorphism. We continue with $\alpha' = (\ell,1^{\ell+i})$, $\alpha'' = (\ell-i, \ell^{d'-3})$ for $i=1,2,...,\ell-1$, showing that the inclusions
    \[
        \ribext{\bullet}{\ell^{d'}} \hookrightarrow \ribext{\bullet}{\ell,1^{2\ell},\ell^{d'-3}}
    \]
    is a quasi-isomorphism. We reduce like that the number of $\ell$ at the end, until there is no more left, to obtain the quasi-isomorphism needed.
\end{proof}
We can now prove the theorem.
\begin{proof}[Proof of the theorem]

    We do the proof by induction on $d$. For $d=0$, letting $c(0) = [] \in \ribext{0}{1^0}$ where $0$ is the sequence with only zero, the result follows from the fact that $\ribext{0}{1^0} = \field[]$ and $\ribext{\bullet}{1^0}$ is concentrated in degree $0$.

    Now we attack the induction step. We separate in three case :
    \begin{itemize}
        \item The case $d \not \equiv 0,1 \modulo \ell$;
        \item The case $d \equiv 1 \modulo \ell$;
        \item And the case $d \equiv 0 \modulo \ell$;
    \end{itemize}
    The case $d \not \equiv 0,1 \modulo \ell$ is just a special case of proposition \ref{prop : ext between some hook and exterior}. In this case, $A_d$ is empty, since if $(a_i)_{i \geq 0} \in A_d$, then $d \equiv a_0 \equiv 0,1 \modulo \ell$. 
    
    
    We now do the case $d \equiv 1 \modulo \ell$. We use the long exact sequence of lemma \ref{lemma : exact triangle dual} with $\alpha' = (1)$ and $\alpha'' = (1^{d-1})$. In this case, $\beta = (2,1^{d-2})$, and proposition \ref{prop : ext between some hook and exterior} show that $\ribext{\bullet}{\beta}$ is acyclic. 
    We conclude that
    \[
        H_*(\pi) : H_*(1^d) \to H_1(1) \otimes H_{*-1}(1^{d-1})
    \]
    is an isomorphism (remember that $\ribext{\bullet}{1}$ is concentrated in degree $1$, where it is $\field [1]$). By the induction hypothesis, $H_\bullet(1^{d-1})$ has basis the $c(a_0,a_1,...)$, where $(a_0,a_1,...) \in A_{d-1}$. Hence, we need to find a cycle in $\ribext{\bullet}{1^d}$ which is send by $\pi$ to $c(a_0,a_1,...)$ (up to addition by a border). We consider $c(a_0+1,a_1,...)$. First, note that since $a_0 \equiv d-1 \equiv 0 \modulo \ell$, and hence $(a_0+1,a_1,...) \in A_d$. Now, $\pi$ sends an element $[\gamma]$ to a non-zero element of $\ribext{\bullet}{1} \otimes \ribext{\bullet}{1^{d-1}}$ if and only if $\gamma_1 = 1$. Now, let $(\gamma_1,\gamma_2,...,\gamma_k)$ be the compositions obtained by concatenation of the compositions who defined $c_0(a_0+1)$, $c_1(a_1)$,...,$c_r(a_r)$ (with $r$ maximal such that $a_r \neq 0$). Then 
    \[
        c(a_0+1,a_2,...) = \sum_{\sigma} (-1)^{\ell(\sigma)} [\gamma_{\sigma^{-1}(1)},...,\gamma_{\sigma^{-1}(k)}]
    \]
    where the sum run over all $(\alpha_1,...,\alpha_r)$-shuffle where $\alpha_i$ is the size of $c_i(a_i)$. In this sum, the partition $(\gamma_{\sigma^{-1}(1)},...,\gamma_{\sigma^{-1}(k)})$ starts by a $1$ if and only if $\sigma(1) = 1$. Thus,
    \[
        \pi (c(a_0+1,a_1,a_2,...)) = \sum_{\sigma \text{ such that } \sigma(1) = 1} (-1)^{\ell(\sigma)} [1] \otimes [\gamma_{\sigma^{-1}(2)},...,\gamma_{\sigma^{-1}(k)}] \ .
    \]
    Note that $(\gamma_2,...,\gamma_k)$ is the compositions obtained by concatenation of $c_0(a_0)$, $c_1(a_1)$,...,$c_r(a_r)$. Thus, this sums is equal to $[1] * c(a_0,a_1,...)$. Hence the inverse of the isomorphism $H_*(\pi)$ is given by
    \[
        \overline{c(a_0,a_1,...)} \mapsto \overline{c(a_0+1,a_1,...)} \ .
    \]
    Since $(a_0,a_1,...) \mapsto (a_0+1,a_1,...)$ is a bijection from $A_{d-1}$ to $A_d$, this shows that the induction hold in this case.

    Finally, we do the case $d \equiv 0 \modulo \ell$.
    This is the most difficult part. We use the long exact sequence of lemma \ref{lemma : exact triangle dual} with $\alpha' = (\ell-1)$ and $\alpha'' = 1^{d-\ell+1}$. 
    We will prove later that the maps $H_*(\Psi)$ are $0$, and hence that we have short exact sequences
    \[
        0 \to H_*(\ell,1^{d-\ell}) \xrightarrow{H_*(\iota)} H_*(\ell-1,1^{d-\ell+1})  \xrightarrow{H_*(\pi)} H_1(\ell-1) \otimes H_{*-1}(1^{d-\ell+1}) \to 0
    \]
    We first use these short exact sequences.
    
    \begin{itemize}
        \item On one hand, consider $(a_0,a_1,a_2,...) \in A_{d - \ell + 1}$. Then $(a_0+\ell-1,a_1,a_2,...) \in A_d$, and as before, we can prove that
        \[
            \pi(c(a_0+\ell-1,a_1,a_2,...)) = [\ell-1] \otimes c(a_0,a_1,a_2,...).
        \]
        Hence $\overline{c(a_0,a_1,a_2,...)} \mapsto \overline{c(a_0+ \ell -1,a_1,a_2,...)}$ defines a section of $H_*(\pi)$.
        \item On the other hand, our induction hypothesis and the lemmas \ref{lemma : going from classical to quantum} and \ref{lemma : going from classical to quantum 2} show that 
        \[
            H_\bullet(\ell,1^{d-\ell}) = \bigoplus_{a \in \underline{A_{d/\ell}}} \field c(0,a_0,a_1,...) \ .
        \]
        where $\underline{A_{d/\ell}}$ is defined as $A_{d/\ell}$ but with $\ell = p$ in the definition (this is the case $q=1$).
    \end{itemize}
    Hence
    \[
        H_\bullet(1^{d}) = \left ( \bigoplus_{a \in A_{d - \ell + 1}} \field c(a_0+\ell-1,a_1,a_2,...) \right ) \oplus \left ( \bigoplus_{a \in \underline{A_{d/\ell}}} \field c(0,a_0,a_1,a_2,...) \right ) \ .
    \]
    Since the map
    \[
        \begin{array}{rcl}
            A_{d - \ell + 1} \sqcup \underline{A_{d/\ell}} & \to & A_d\\
            a & \mapsto & \left \{ \begin{array}{cl}
                (a_0 + \ell-1,a_1,a_2,...) & \text{if } a \in A_{d - \ell + 1}, \\
                (0,a_0,a_1,a_2,...) & \text{if } a \in \underline{A_{d/\ell}},
            \end{array} \right .
        \end{array}
    \]
    is a bijection, this proves the result for this case. Hence, to end this proof, we just need to show that the map $H_*(\Psi)$ in the long exact sequence of lemma \ref{lemma : exact triangle dual} with $\alpha' = (\ell-1), \alpha'' = (1^{d-\ell+1})$ are all zero. In other words, we need to prove that $\Psi([\ell-1] \otimes c(a))$ is a boundary for all $a \in A_{d-\ell+1}$. 
    By definition, $c(a)$ is a sum of scalar multiple of $[\gamma]$ with the $\gamma$ being composition starting by one of the following integer : $1,\ell-1,\ell,(p-1)\ell,p\ell,(p-1)p\ell,...$. Hence, $\Psi([\ell-1] \otimes c(a))$ is a sum of scalar multiple of $[\gamma]$ with $\gamma$ starting by $\ell$ or an integer not divisible by $\ell$. In the first case, the coefficient between the $\gamma$ is a multiple of $\binom{\ell}{1}_{q^-2} = 0$ by definition of $\Psi$. This show that $\Psi([\ell-1] \otimes c(a))$ is inside
    \[
        \bigoplus_{\substack{\gamma \in \refinement{}{\ell,1^{d-\ell}}\\ \text{such that } \ell \text{ does not divide } \gamma}} \field [\gamma] \ .
    \]
    It is in fact a subcomplex of $\ribext{\bullet}{\ell,1^{d-\ell}}$, which is complementary to $\ribext{\bullet}{\ell^{d/\ell}}$. This is a subcomplex since if $\gamma_t$ is not divisible by $\ell$, but $\gamma_t + \gamma_{t+1}$ is, then $\qbinom{\gamma_t + \gamma_{t+1}}{\gamma_t} = 0$. The fact that it is complementary is trivial ($\ribext{\bullet}{\ell^{d/\ell}}$ is spanned by the $[\gamma]$ with $\gamma$ refinement of $(\ell,1^{d-\ell})$ which are divisible by $\ell$). But since the injection $\ribext{\bullet}{\ell^{d/\ell}} \hookrightarrow \ribext{\bullet}{\ell,1^{d-\ell}}$ is a quasi-isomorphism, this shows that this subcomplex is acyclic, and hence $\Psi([\ell-1] \otimes c(a))$ is a boundary.
\end{proof}

In conclusion, we go back to our original aim : computing $\Ext{}{}$-groups in the category of $\field \GL{d}$-module. Remember that $\projbound{\qsymmetric{d}} = 2(\ell - 1)$ for $d \geq \ell$ and $\injbound{\qexterior{d}} = \projbound{\qexterior{d}} = \ell - 1$ (since $\qdual{\qexterior{}} \simeq \qexterior{d}$). Hence we can use corollary \ref{cor : ext between symmetric and exterior} to compute $\Ext{\field \GL{d}}{*}(\beta(\qsymmetric{d}), \beta(\qexterior{d}))$ in degree $* < 3(\ell - 1)$. Since $\beta(\qsymmetric{d}) = \symflag{d}$ and $\beta(\qexterior{d}) = \field$ (the trivial representation), we obtain the following computations.

\begin{example}\label{exam : coh 1}
    If $d \not \equiv 0,1 \modulo \ell$, then
    \[
        \Ext{\field \GL{d}}{*}(\symflag{d}, \field) = 0 \quad \text{for } * \leq 3(\ell-1).
    \]
    Using the duality $\qdual{-}$, we can also interpret this in term of group cohomology :
    \[
        H^*(\GL{d},\dividedmod{d}) = 0 \quad \text{for } * \leq 3(\ell-1).
    \]
    In the rest of this example, we state the results in this form.
    
    
\end{example}

\begin{example}\label{exam : coh 2}
    When $d = k \ell +1$, corollary \ref{cor : ext between symmetric and exterior} implies that $\Ext{\quantumfunctor{}}{*}(\qsymmetric{k\ell+1},\qexterior{k \ell+1}) \simeq \Ext{\quantumfunctor{}}{*}(\qsymmetric{k\ell},\qexterior{k \ell})$. Hence,
    \[
        H^*(\GL{k\ell+1},\dividedmod{k \ell+1}) \simeq H^*(\GL{k\ell},\dividedmod{k \ell}) \quad \text{for } * \leq 3(\ell-1).
    \]
\end{example}


\begin{example}\label{exam : coh 3}
    Using corollary \ref{cor : ext between symmetric and exterior}, we obtain the following equalities.
    \[
        \dim H^*(\GL{\ell},\dividedmod{\ell}) = \left \{ \begin{array}{cl}
            1 & \text{if } *=\ell-2 \text{ or } * = \ell -1 , \\
            0 & \text{if } * < \ell - 2 \text{ or } \ell \leq * < 3\ell-3.
        \end{array} \right .
    \]
    \[
        \dim H^*(\GL{2\ell},\dividedmod{2\ell}) = \left \{ \begin{array}{cl}
            1 & \text{if } *=2\ell-4 \text{ or } * = 2\ell - 3 , \\
            0 & \text{if } * < 2\ell - 4 \text{ or } 2\ell -2 \leq * \leq 3\ell-3.
        \end{array} \right .
    \]
    \[
        \dim H^*(\GL{3\ell},\dividedmod{3\ell}) = \left \{ \begin{array}{cl}
            1 & \text{if } *=3\ell-6 \text{ or } * = 3\ell - 5 , \\
            0 & \text{if } * < 3\ell - 6 \text{ or } * = 3\ell - 4, 3\ell - 3 .
        \end{array} \right .
    \]
    Moreover,
    \[
        \dim H^*(\GL{k\ell},\dividedmod{k\ell}) = 0 \quad \text{for } * \leq 3\ell-3 \text{ when } \left \{ \begin{array}{c}
            k \geq 4, \ \ell \geq 6 \\
            k \geq 5, \ \ell \geq 4 \\
            k \geq 7, \ \ell \geq 3
        \end{array} \right .
    \]
    Note that with our hypothesis on $q$, we always have $\ell \geq 3$ and the characteristic $p$ of $\field$ is $\neq 2$. When $p=3$, we also obtain $\dim H^{3\ell-2}(\GL{3\ell}, \dividedmod{3\ell}) \geq 1$.
\end{example}

\begin{example}\label{exam : coh 4}
    For low value of $\ell$, we can obtain more results.
    \begin{itemize}
        \item For $\ell = 6$, we have $\dim H^{16}(\GL{24},\dividedmod{24}) \geq 1$.
        \item For $\ell = 5$, we have
        \[
            \dim H^{13}(\GL{20},\dividedmod{20}) \geq 1 \quad \text{and} \quad \dim H^{*}(\GL{20},\dividedmod{20}) = \left \{\begin{array}{cl}
                1 & \text{if } *=12, \\
                0 & \text{if } *<12. 
            \end{array} \right .
        \]
        \item For $\ell = 4$, we have
        \[
            \dim H^{10}(\GL{20},\dividedmod{20}) \geq 1 \quad \text{and} \quad\dim H^{*}(\GL{16},\dividedmod{16}) = \left \{\begin{array}{cl}
                1 & \text{if } *=8,9, \\
                0 & \text{if } *<8. 
            \end{array} \right .
        \]
        \item Finally, for $\ell = 3$,
        \begin{align*}
            \dim H^{*}(\GL{12},\dividedmod{12}) & = \left \{\begin{array}{cl}
                1 & \text{if } *=4,5, \\
                0 & \text{if } *<4 \text{ or } * = 6,
            \end{array} \right . \\ 
            \dim H^{*}(\GL{15},\dividedmod{15}) & = \left \{\begin{array}{cl}
                1 & \text{if } *=5,6, \\
                0 & \text{if } *<5. 
            \end{array} \right . \\
            \dim H^{*}(\GL{18},\dividedmod{18}) & = \left \{\begin{array}{cl}
                1 & \text{if } *=6, \\
                0 & \text{if } *<6,
            \end{array} \right . \\ 
            \dim H^{7}(\GL{18},\dividedmod{18}) & \geq 1 \ .
        \end{align*}
            
    \end{itemize}
\end{example}

Using the exponentiality of $\qexterior{*}$, we can also compute some cohomology groups of the form $H^*(\GL{d},\dividedmod{\nu})$ with various composition $\nu$ of $d$.

\bibliography{biblio.bib}
\bibliographystyle{plain}

\end{document}